\documentclass[11pt,a4paper,reqno]{amsart}
\usepackage{amssymb}

\usepackage{xspace}
\usepackage[foot]{amsaddr}

\usepackage{times}
\usepackage[shortlabels]{enumitem}
\usepackage[hidelinks]{hyperref}

\newenvironment{bsmallmatrix}{\left[\begin{smallmatrix}}{\end{smallmatrix}\right]}
\newcommand{\Mid}{\;\middle|\;}

\newtheorem{theorem}{Theorem}[section]
\newtheorem{prop}[theorem]{Proposition}

\theoremstyle{definition}
\newtheorem{definition}[theorem]{Definition}
\newtheorem*{example}{Example}

\theoremstyle{remark}
\newtheorem*{remark}{Remark}

\newcommand{\XX}{{\mathfrak  S}}
\newcommand{\Z}{{\mathbb Z}}
\newcommand{\und}[1]{\underline{#1}}
\newcommand{\aut}[1]{{\sf Aut}(#1)}
\newcommand{\F}{{\mathbb F}}
\newcommand{\cG}{\mathcal{G}}
\newcommand{\cH}{\mathcal{H}}
\newcommand{\cK}{\mathcal{K}}

\newcommand{\sym}[1]{{\sf Sym }(#1)}
\newcommand{\cS}{{\mathcal S}}
\newcommand{\cA}{{\mathcal A}}
\newcommand{\half}{\frac{1}{2}}

\newcommand{\cJ}{\mathcal J}
\newcommand{\jaut}{\mathsf{JAut}}
\newcommand{\aaut}{\mathsf{AAut}}
\newcommand{\taut}{\mathsf{TAut}}
\newcommand{\gen}{autonomous\xspace}
\newcommand{\ingen}{non-autonomous\xspace}

\newcommand{\CC}{coherent configuration\xspace}
\newcommand{\CCs}{coherent configurations\xspace}
\newcommand{\AS}{association scheme\xspace}
\newcommand{\ASs}{association schemes\xspace}
\newcommand{\JC}{Jordan configuration\xspace}

\newcommand{\JS}{Jordan scheme\xspace}
\newcommand{\JSs}{Jordan schemes\xspace}

\newcommand{\WL}{\operatorname{WL}}

\newcommand{\restr}{\mathord{\upharpoonright}}

\author[M.\,Muzychuk]{Mikhail Muzychuk}
\address[M.\,Muzychuk]{Ben Gurion University of the Negev, Beer Sheva, Israel}
\email{muzychuk@bgu.ac.il}

\author[Ch.\,Pech]{Christian Pech}
\address[Ch.\,Pech]{Institute of Mathematics\\Czech Academy of Sciences\\\v{Z}itn\'a 25\\ 115\,67 Praha 1\\ Czech Republic}
\email{pech@math.cas.cz}
\thanks{The research of the second author was supported by GA~\v{C}R (Czech Science Foundation) grant EXPRO 20-31529X}

\author[A.\,Woldar]{Andrew Woldar}
\address[A.\,Woldar]{Villanova University, Villanova, PA, USA}
\email{andrew.woldar@villanova.edu}

\title{Classification of thin Jordan schemes}


\date{}

\begin{document}

	\begin{abstract}
		Jordan schemes generalize association schemes in a manner similar to how Jordan algebras generalize associative algebras. It is well known that association schemes of maximum rank are in one-to-one correspondence with groups (so-called thin schemes). In this paper, we classify Jordan schemes of maximum rank-to-order ratio and show that regular Jordan schemes of such type correspond to a special class of Moufang loops known as ring-alternative loops. 
	\end{abstract}
	\maketitle

    \section{Introduction}

    In 1959, Shah proposed \cite{Sha59} to use the Jordan matrix product as an algebraic tool for the analysis of block designs. This was a far-reaching generalization of the concept of an association scheme introduced in the same year by Bose and Mesner \cite{BosMes59}. The concept was further developed and led to a flourishing research area called ``Algebraic Combinatorics" by Bannai and Ito in their famous monograph~\cite{BanIto84}. Nevertheless, Shah's idea went virtually unnoticed until Bailey published her book~
    \cite{Bai04} where she analyzed the ``first" properties of objects arising from Shah's idea. In particular, she showed that the symmetrization of an association scheme always produces objects introduced by Shah. In 2001, Cameron called these objects Jordan schemes and posed the question, ``Is every symmetric Jordan scheme a symmetrization of an association scheme?'' In 2019,  Klin, Muzychuk and Reichard gave a negative answer to Cameron's question and started a systematic investigation of Jordan schemes and, more generally, coherent Jordan configurations \cite{KliMuzRei19}.

    It is well known that all finite groups form a special subclass of association schemes, called thin schemes \cite{Zie05}. Such schemes may be characterized as those that have a maximum rank-to-order ratio among all association schemes. In this paper, we classify all Jordan schemes that attain maximum rank-to-order ratio. We further show that regular Jordan schemes with this property are in a one-to-one correspondence with a special class of Moufang loops. We direct the reader to \cite{Bru71} for detailed information about various types of loops, e.g. Moufang loops, Bol loops, etc.

    The paper is organized as follows. Section 2 contains all necessary definitions and basic facts about coherent Jordan configurations and Jordan schemes. Section 3 is the most central part of our paper. There we show that the rank-to-order ratio of a \JS is bounded by $3/2$ with equality occurring if and only if the scheme is thin and nonregular. All such Jordan schemes are classified in Subsection 3.1. Subsection 3.2 treats thin regular Jordan schemes. There we establish a one-to-one correspondence between such schemes and ring-alternative Moufang loops (RA-loops). Finally, in Section 4 we show that Jordan schemes corresponding to RA-loops are \gen, that is to say, they are not algebraic fusions of coherent configurations. 

Thus we describe in Section 4 the first infinite family of \gen Jordan schemes. Previously, the only known \gen Jordan scheme was the one on $15$ points described in \cite{MuzReiKli22} (see also \cite[Section 4]{KliMuzRei19}) related to one of the two Siamese $\mathrm{STS}(15)$ constructed in \cite[Section 4]{KliReiWol05}, see also \cite{KliReiWol09}. The fact that this \JS is \gen was established using the COCO2P package \cite{COCO2P} in GAP \cite{GAP4} and Ziv-Av's catalog of small coherent configurations \cite{Ziv18}.
    
    \section{Background Material}
    
	\subsection{Binary relations and matrices} Let $\Omega$ be a nonempty finite set. A \emph{binary relation} on $\Omega$ is an arbitrary subset $s\subseteq \Omega\times\Omega$. We refer to the set ${\omega} s:=\{\beta\in\Omega \mid (\omega,\beta)\in s\}$ as the \emph{$s$-neighborhood} of $\omega\in\Omega$. A relation is called \emph{regular} if $|\omega s|$ does not depend on $\omega$. In such case, we refer to the constant $n_s:=|\omega s|$ as the \emph{valency} of $s$.   
	If $s$ is a regular relation of valency one, then ${\omega}s$ contains a unique element which we denote as $s(\omega)$, i.e.\ ${\omega} s=\{s(\omega)\}$. Thus $s$ is a function on $\Omega$. All functions in this paper, including permutations from $\sym{\Omega}$, will be considered as binary relations. 
	
	Every relation $s$ admits a \emph{transposed} 
	relation $s^t=\{(\alpha,\beta)\in\Omega^2 \mid(\beta,\alpha)\in s\}$.  Throughout, we will denote the unary operator that transposes relations by $\tau$, that is, $\tau:s\mapsto s^t$. Following~\cite{ChePon24} we call a relation $s\subseteq \Omega\times\Omega$ {\it thin} if $|\omega s|\leq 1\geq |\omega s^t|$ holds for each $\omega\in\Omega$.
    
    We write $ab$ for the relational product of $a,b\subseteq\Omega^2$ and $a\star b = ab\cup ba$ for their \emph{Jordan product}. For $\Delta\subseteq \Omega$,  
    the \emph{diagonal relation} $1_\Delta$ is defined by $1_\Delta:=\{(\delta,\delta) \mid\delta\in\Delta\}$.
	The \emph{adjacency matrix} of $s\subseteq\Omega^2$ is the square matrix, denoted $\und{s}$, whose rows and columns are labeled by the elements of $\Omega$ and whose $(\alpha,\beta)$-entry $\und{s}_{\alpha,\beta}$ equals one if $(\alpha,\beta)\in s$ and zero otherwise.
    
	The algebra of all $\Omega\times\Omega$ matrices over a field  $\mathbb{F}$ will be denoted as $M_\Omega(\mathbb{F})$. It is always assumed that $\mathsf{char}(\F)\neq 2$. The standard product of two matrices $A,B\in M_\Omega(\mathbb{F})$ will be written as $A\cdot B$ or $AB$, while their Jordan product $\half (AB+BA)$ will be denoted by $A\star B$. The Schur-Hadamard (or entry-wise) product will be denoted as $A\circ B$. The identity matrix and all-ones matrix will be denoted as $I_\Omega$ and $J_\Omega$, respectively.
    The linear span of any set of matrices ${\mathcal S}\subseteq M_\Omega(\F)$ will be denoted by $\mathbb{F}[\mathcal{S}]$. In the case where $\mathcal{S}$ is given as a family $\{A_i\}_{i\in I}$, we also write $\langle A_i\rangle_{i\in I}$ for $\mathbb{F}[\mathcal{S}]$. 
    
	\subsection{Rainbow algebras}
	Consider a vector subspace $\cA\subseteq M_\Omega(\mathbb{F})$ which is closed with respect to $\tau$ and $\circ$ and contains $I_\Omega, J_\Omega$.
    Every such subspace $\cA\subseteq M_\Omega(\mathbb{F})$ has a uniquely determined \emph{standard basis}, namely the basis consisting of all minimal $\circ$-idempotents of the algebra $(\cA,\circ)$. These are $\{0,1\}$-matrices $A_1,\dots,A_r\in\cA$ that are pairwise $\circ$-orthogonal. Each $A_i$ is the adjacency matrix of a unique binary relation $s_i\subseteq\Omega^2$, i.e.\ $A_i = \und{s}_i,i=1,\dots,r$. The set $S=\{s_1,\dots,s_r\}$ forms a partition of $\Omega^2$ that satisfies the following two conditions: 
	\begin{enumerate}[(a)]
			\item $1_\Omega$ is a union of some relations  of $S$;
			\item $S^t = S$, where $S^t:=\{s^t \mid s\in S\}$;
	\end{enumerate}
	Following \cite{Hig90} we call a pair $\mathfrak{S}=(\Omega,S)$  a \emph{rainbow} if it satisfies these conditions. 
	The elements of $S$ are called \emph{basic relations} of the rainbow. The cardinalities $|\Omega|$ and $|S|$ are called the \emph{order} and \emph{rank} of the rainbow, respectively. Observe that a rainbow has rank one if and only if $|\Omega|=1$. For $|\Omega|\ge 2$ one has $|S|\geq 2$, where equality is attained if and only if $S=\{1_\Omega,\,\Omega^2\setminus 1_\Omega\}$. Such rainbows
	are called \emph{trivial}. The \emph{relation matrix} of a rainbow ${\mathfrak S} = (\Omega,S)$ is the $\Omega\times\Omega$ matrix $A({\mathfrak S})$ defined via $A({\mathfrak S})_{\alpha,\beta} = S(\alpha,\beta)$ where $S(\alpha,\beta)$ is the unique basic relation containing $(\alpha,\beta)\in\Omega^2$.
    
    A rainbow of maximum rank $|\Omega|^2$ is called \emph{discrete}. 
    Each basic relation of a discrete rainbow contains a unique ordered pair.
	
	In what follows, we call a subspace $\cA\subseteq M_\Omega(\mathbb{F})$ a \emph{rainbow algebra} if it contains $I_\Omega,J_\Omega$ and is closed with respect to $\tau$ and $\circ$. We say $\cA$ is \emph{symmetric} if every $A\in \cA$ symmetric, i.e.\ $A^t=A$ for every $A\in \cA$. We call $\cA$ \emph{homogeneous} if $\langle I_\Omega \rangle\circ\cA\subseteq \langle I_\Omega \rangle$ and \emph{regular} if $\langle J_\Omega\rangle\cdot\cA\subseteq \langle J_\Omega\rangle$. 
	
	Given a rainbow $\mathfrak{S}=(\Omega,S)$, the linear span $\mathbb{F}[S]:=\langle \und{s}\rangle_{s\in S}\subseteq M_\Omega(\mathbb{F})$ is a rainbow algebra. Thus, there exists a one-to-one correspondence
	between rainbows and rainbow subalgebras of $M_\Omega(\mathbb{F})$. The statement below collects some properties of rainbow algebras.
	\begin{prop}
    Let $\mathfrak{S}=(\Omega,S)$ be a rainbow, and let $\cA=\langle \und{s}\rangle_{s\in S}$ be the corresponding rainbow algebra. Then
		\begin{enumerate}[(a)]
			\item $\cA$ is symmetric iff all basic relations of $\mathfrak{S}$ are symmetric;
			\item $\cA$ is homogeneous iff $1_\Omega\in S$;
			\item $\cA$ is regular iff all basic relations of $\mathfrak{S}$ are regular.
		\end{enumerate}
	\end{prop}
		In what follows, we call a rainbow \emph{symmetric}, \emph{homogeneous}, or \emph{regular} depending on which of those properties its corresponding algebra satisfies. Note that regularity of a rainbow always implies homogeneity but not vice versa.
        A rainbow will be called \emph{thin} if all of its basic relations are thin.

    Given two rainbows $\mathfrak{S}=(\Omega,S)$ and $\mathfrak{S'}=(\Omega,S')$, we say that $\mathfrak{S}$ is a \emph{fusion} of $\mathfrak{S'}$ (equivalently, $\mathfrak{S'}$ is a \emph{fission} of $\mathfrak{S}$) if $\mathbb{F}[S]\subseteq\mathbb{F}[S']$. In this case, we also write $S\sqsubseteq S'$.
		
    Two rainbows $\mathfrak{S}=(\Omega,S)$ and $\mathfrak{S'}=(\Omega',S'
    )$ are said to be \emph{(combinatorially) isomorphic} if there exists a bijection $f:\Omega\rightarrow \Omega'$ satisfying $f(S)=S'$, where $f(S):=\{f(s) \mid s\in S\}$ and $f(s):=\{(f(\alpha),f(\beta)) \mid  (\alpha,\beta)\in s\}$.
    
	\subsection{Coherent Jordan algebras}
		
	A rainbow algebra is called a \emph{coherent algebra} (see \cite{Hig90}) if it is closed with respect to standard matrix multiplication. Similarly, a rainbow algebra will be called a \emph{coherent Jordan algebra} (coherent J-algebra, for short) if it is closed with respect to Jordan multiplication $\star$. 
	
	Let $\mathfrak{S}=(\Omega,S)$ be a rainbow, and let $\cA:=\langle \und{s}\rangle_{s\in S}$ be its corresponding rainbow algebra. We say $\mathfrak{S}$ is a \emph{coherent configuration} (resp.\ \emph{J-configuration}) if $\cA$ is a coherent algebra (resp.\ J-algebra). In both cases, the algebra $\cA$ will be referred to as the \emph{adjacency algebra} of $\mathfrak{S}$.

    If $\mathfrak{S}=(\Omega,S)$ is a coherent configuration, then for any pair $s,r\in S$ we have $\und{s}\cdot\und{r} = \sum_{t} p_{s,r}^t\und{t}$ for suitable nonnegative integers $p_{s,r}^t$. We refer to these $p_{s,r}^t$
    as {\it structure constants} of $\mathfrak{S}$. Due to their combinatorial meaning $p_{s,r}^t=|\alpha s\cap \beta r^t|, (\alpha,\beta)\in t$, they are also called \emph{intersection numbers} of $\mathfrak{S}$.

Similarly, one can define intersection numbers $p_{\{s,r\}}^t$ for a coherent J-configuration as follows: $p_{\{s,r\}}^t=|\alpha s\cap \beta r^t| + |\alpha r\cap \beta s^t|, (\alpha,\beta)\in t$.
These also appear as structure constants for the Jordan product $\und{s}\star\und{r} = \frac{1}{2}\sum_{t} p_{\{s,r\}}^t\und{t}$.
(Note that our notation $p_{\{s,r\}}^t$ reflects the fact that $\star$ is commutative.)

	Recall that a homogeneous coherent configuration is called an \emph{association scheme} (see \cite{Zie05}). Following this pattern, we will call a homogeneous coherent J-configuration a \emph{Jordan scheme}. While it is well known that association schemes are always regular, Jordan schemes do not possess this property in general. It was shown in~\cite{MuzReiKli22} that a symmetric \JS is always regular. Regarding nonsymmetric Jordan schemes, they are not necessarily regular, but have a very special structure --- see \cite[Proposition 2.2]{MuzReiKli22}. Examples of nonregular Jordan schemes also appear in Subsection~\ref{nonreg} of this paper.
	\subsection{Coherent J-closure}
    It follows immediately from the definition of a coherent J-algebra that the intersection of any set of coherent J-algebras defined on $\Omega$ is again a coherent J-algebra on $\Omega$. This observation leads to the notion of the \emph{coherent J-closure} $J(\mathcal{S})$ of a set $\mathcal{S}\subseteq M_\Omega(\mathbb{F})$ of matrices as the intersection of all coherent J-algebras containing $\mathcal{S}$. 
    
    The coherent J-closure generalizes the well-known concept of the \emph{coherent closure} of $\mathcal{S}$ which is the intersection of all coherent algebras containing $\mathcal{S}$. It is often called the WL-closure in honor of Weisfeiler and Leman \cite{WeiLem68}, who designed an algorithm for its computation. For this reason, we will denote this closure by $\WL(\mathcal{S})$. Observe that we always have $J(\mathcal{S})\subseteq \WL(\mathcal{S})$.
    If $S$ is a set of binary relations on $\Omega$, then we write $J(S)$ and $\WL(S)$ for the underlying rainbows of $J(\{\und{s}\}_{s\in \mathcal{S}})$ and $\WL(\{\und{s}\}_{s\in \mathcal{S}})$, respectively.
    
	\subsection{Algebraic fusions of J-configurations}

	Given a J-configuration $\XX=(\Omega,S)$, a permutation $\alpha\in\sym{S}$ is called an \emph{algebraic J-automorphism} of $\XX$ if its linear extension to $\mathbb{C}[S]$ is an automorphism of the Jordan algebra $(\mathbb{C}[S],\star)$. All algebraic J-automorphisms of $\XX$ form a subgroup of $\sym{S}$ that will be denoted by $\jaut(\XX)$. Note that the transpose permutation $\tau: a\mapsto a^t$ is always an algebraic J-automorphism of $\XX$. The group $\jaut(\XX)$ consists of all permutations of $\sym{S}$ that preserve the tensor of structure constants of Jordan multiplication with respect to the standard basis. In other words, a permutation $u\mapsto u'$ of $S$ belongs to $ \jaut(\XX)$ if and only if  
	\[
    \forall_{u,v,w\in S}:\  p_{\{u',v'\}}^{w'}=p_{\{u, v\}}^{w}.
    \]

    The following statement collects certain properties of subgroups of $\jaut(\XX)$.
    
	\begin{prop}
		Each of the following holds:
		\begin{enumerate}[(a)]
			\item The transpose permutation $\tau$ belongs to the center of $\jaut(\XX)$;
			\item For any $\Phi\leq \jaut(\XX)$ and $s\in S$ the relation $\Phi(s):=\bigcup_{\phi\in\Phi}\phi(s)$ is either symmetric or anti-symmetric;
			\item For any $\Phi\leq \jaut(\XX)$, $\bigl(\Omega,\{\Phi(s)\mid s\in S\}\bigr)$ is a \JC.
		\end{enumerate}
	\end{prop}
	\begin{proof}
		{\sl Proof of \ref{conj_a}:} One readily checks that for any pair $r,s\in S$, 
		\[
        s=r^t\iff (\und{s}\star\und{r})\circ I_\Omega\neq O_\Omega  \mbox{ (the $\Omega\times\Omega$ zero matrix}).
        \]
		Since any $\phi\in\jaut(\XX)$ preserves $I_\Omega$, we conclude that
		\[
		(\und{s}\star\und{r})\circ I_\Omega\neq O_\Omega\iff (\und{\phi(s)}\star\und{\phi(r)})\circ I_\Omega\neq O_\Omega,
		\]
		thereby implying
		that $s=r^t\iff \phi(s)=\phi(r)^t$. 
		\smallskip
		
		{\sl Proof of \ref{conj_b}:}   
		It follows from part (a) that
		\[
		(\Phi(s))^t = \bigcup_{\phi\in\Phi}\phi(s)^t = \bigcup_{\phi\in\Phi}\phi(s^t) = \Phi(s^t).
		\]
		If $s^t = \phi(s)$ for some $\phi\in\Phi$, then $\Phi(s^t) = \Phi(\phi(s))=\Phi(s)$, which implies $\Phi(s)$ is symmetric.
		If, on the other hand, $s^t\neq \phi(s)$ for every $\phi\in\Phi$, then $\psi(s^t)\neq \phi(s)$ and consequently $\psi(s^t)\cap \phi(s) = \emptyset$ for every pair $\psi,\phi\in\Phi$. Therefore 
		$\Phi(s^t)\cap \Phi(s) = \bigcup_{\psi,\phi\in\Phi} (\psi(s^t)\cap\phi(s)) = \emptyset$, whereby  $\Phi(s)$ is anti-symmetric in this case.
		\smallskip
		
		{\sl Proof of \ref{conj_c}:} Denote $\cA:=\mathbb{C}[S]$ and extend each $\phi\in\Phi$ on $\cA$ linearly via $\phi(\und{s}) := \und{\phi(s)}$. We will use the same notation for the induced linear operator on $\cA$. Then
	    \[
        \phi(X\star Y) = \phi(X)\star\phi(Y),\;
	    \phi(X\circ Y) = \phi(X)\circ\phi(Y),\; \phi(X^t) = \phi(X)^t. 
        \]
	    Since $\phi(I_\Omega)=I_\Omega$ and $\phi(J_\Omega)=J_\Omega$,  the  subspace ${}^\Phi\!\!\cA = \{X\in\cA \mid  \forall_{\phi\in\Phi}\ \phi(X) = X\}$ is a coherent Jordan algebra. 	    
	    It is easy to verify that the matrices $\{ \und{\Phi(s)}\}_{s\in S}$  form the standard basis of ${}^\Phi\!\!\cA$.
    \end{proof}
	It follows from above that given any $\Phi\leq \jaut(\XX)$, the pair ${}^{\Phi}{\XX}:=(\Omega,\{\Phi(s)\}_{s\in S})$ is also a \JC which we will refer to as an \emph{algebraic fusion} of $\XX$.
	
	\subsection{Constructing Jordan configurations from coherent ones}
	
	If $\XX=(\Omega,S)$ is a 
	\CC, then it is also a \JC. Recall that a permutation $\phi$ of $S$ is called an \emph{algebraic automorphism} (\emph{anti-automorphism}) of $\XX$ if its linear extension to $\mathcal{A}:=\langle \und{s}\rangle_{s\in S}$ is an automorphism (anti-automorphism) of $(\mathcal{A},\cdot)$. The group $\aaut(\XX)$ of all algebraic automorphisms of $\XX$ 
	is contained in $\jaut(\XX)$. Also, $\tau\in\jaut(\XX)$. Every anti-automorphism of $\XX$ is a composition of $\tau$ and some $\phi\in\aaut(\XX)$. All automorphisms and anti-automorphisms of $\XX$ form a subgroup of $\jaut(\XX)$ which will be denoted by $\taut(\XX)$. If $\XX$ is commutative, then $\aaut(\XX)=\taut(\XX)=\jaut(\XX)$. If however $\XX$ is not commutative, then  $\tau\not\in\aaut(\XX)$ and $[\taut(\XX)\!:\!\aaut(\XX)]=2$ whence $\taut(\XX) = \aaut(\XX)\times\langle \tau\rangle$.
	In general, the inclusion $\taut(\XX)\leq\jaut(\XX)$ may be proper. The smallest examples of this are $\mathrm{CC}(8,31)$ and $\mathrm{CC}(8,32)$ of order $8$. (Enumeration of \CCs here is taken from Ziv-Av's catalog \cite{Ziv18} of small coherent configurations.) 
	
	If $\XX=(\Omega,S)$ is a thin association scheme, then $S$ is a regular subgroup of $\sym{\Omega}$. In this case a permutation $\phi\in\sym{S}$ belongs to $\jaut(\XX)$ if and only if it satisfies $\phi(st)\in\{\phi(s)\phi(t),\phi(t)\phi(s)\}$ for all $s,t\in S$. It was shown in \cite{Sco57} that any such map is either an automorphism or anti-automor\-phism.
	Thus for thin \ASs one always has $\taut(\XX)=\jaut(\XX)$.
	
	We call a \JC $\mathfrak{S}=(\Omega,S)$ \emph{\ingen} if it is an algebraic fusion  of a coherent configuration $\mathfrak{C}=(\Omega,C)$ via a subgroup $\Phi\leq \jaut(\mathfrak{C})$. A \JC will be called \emph{\gen} otherwise, that is, if it cannot be obtained as an algebraic fusion of some \CC sharing the same point set. 
    
    Note that in \cite{MuzReiKli22} Jordan schemes were divided into two types: \emph{proper} and \emph{improper}. An improper Jordan scheme is one that can be realized as a symmetrization of an \AS. Otherwise it is proper. Under this definition, any nonsymmetric \JS is automatically proper. 
    Regarding symmetric Jordan schemes, there are proper ones that are \ingen.  We are able to show that there are no examples of proper non-autonomous Jordan schemes of order less than $16$. In fact, the smallest examples of such schemes known to us are of order $24$ as described below.
    
    \begin{example}
    Computer calculations show that each of $\text{AS}(24,526)$, $\text{AS}(24,597)$, $\text{AS}(24,669)$ has two algebraic fusions that are proper symmetric Jordan schemes. Hence, it follows that these schemes are non-autonomous. Enumeration of \ASs here is taken from Hanaki and Miyamoto's   catalog \cite{HanMiy03} of small association schemes. These are also accessible in GAP \cite{GAP4} as well as in the COCO2P package \cite{COCO2P}. 
    \end{example}
    It appears that the proper \JSs constructed in \cite{MuzReiKli22} are also autonomous, but we have yet to prove this. We have however verified this for the smallest example of order $15$ constructed there.
	
	\section{Thin Jordan schemes}

    It is well known that each basic graph of an association scheme $(\Omega,S)$ is regular. Therefore we always have that $|S|\leq |\Omega|$ where equality holds if and only if each basic relation $s\in S$ is a permutation on $\Omega$. Such a scheme is called thin and it is not difficult to check that an association scheme $(\Omega,S)$ is thin if and only if $S$ is a regular subgroup of $\sym{\Omega}$. Since every group has a regular representation, one can consider groups as a particular case of association schemes. 

    Although symmetric Jordan schemes are always regular, this is no longer true of nonsymmetric ones as the following simple example illustrates.

    \begin{example} The full matrix algebra $M_\Omega(\F)$ is a discrete coherent algebra of rank $|\Omega|^2$. In the case of $\Omega=\{1,2\}$ the mapping  
    $\phi: \begin{bmatrix}
       a & b\\
       c & d\\
    \end{bmatrix}\rightarrow \begin{bmatrix}
       d & b\\
       c & a\\
    \end{bmatrix}$ is an anti-automorphism. Therefore, the subspace $\cA := {}^{\!\langle \phi\rangle\!} M_{\{1,2\}}(\F) = \left\{\begin{bmatrix}
       a & b\\
       c & a\\
    \end{bmatrix}\,\vline \, a,b,c\in\F\right\}$ is a coherent J-algebra which has order 2 and rank 3. The corresponding  \JC   $(\Omega,\{s_0,s_1,s_2\})$, where $s_0=I_\Omega$, $s_1=\{(1,2)\}$ and $ s_2=\{(2,1)\}$, is homogeneous but not regular, and its rank-to-order ratio $|S|/|\Omega|$ is $3/2$. 
\end{example}
    
    The following assertion shows that this is the maximum value of a rank-to-order ratio for any nonregular \JS. 

    \begin{prop}
        Let $(\Omega,S)$ be a nonregular \JS. Then $|S|/|\Omega|\leq 3/2$ where equality holds if and only if all basic relations are thin.  
    \end{prop}
    \begin{proof}
         By \cite[Proposition 2.2]{MuzReiKli22} there exists a partition $\Omega = \Omega_1\cup \Omega_0$ with $|\Omega_1|=|\Omega_0|$
such that every regular relation is contained in $\Omega_0^2\cup\Omega_1^2$ and each nonregular one is contained in either $\Omega_0\times\Omega_1$ or $\Omega_1\times\Omega_0$. Thus $S$ splits into a disjoint union $S=S_r\cup S_{0,1}\cup S_{1,0}$ where $S_r$ stands for  all regular relations while
$S_{i,1-i}:=\{s\in S \mid \subseteq \Omega_i\times\Omega_{1-i}\}$. 

Define $n_s:=\half p_{\{s,s^t\}}^{1_{\Omega}}$ for each $s\in S_{i,1-i}$. Then $|s|=n_s |\Omega_i|$, and it follows from $\sum_{s\in S_{i,1-i}} |s|=|\Omega_i|\cdot |\Omega_{1-i}|$ that $\sum_{s\in S_{i,1-i}} n_s=|\Omega|/2$. Since all summands in the latter sum are nonnegative integers, we conclude $|S_{i,1-i}|\leq |\Omega|/2$ for $i=0,1$.

Since a basic relation $s$ is regular if and only if $s\subseteq\Omega_0^2\cup\Omega_1^2$, we can write
\[
|\Omega_0^2\cup\Omega_1^2| = \sum_{s\in S_r} |s|\iff \frac{|\Omega|^2}{2}=\sum_{s\in S_r} n_s|\Omega|\iff \sum_{s\in S_r} n_s = \frac{|\Omega|}{2}.
\]
Arguing as before we conclude that $|S_r|\leq |\Omega|/2$.
Combining together all of the above, we obtain 
\[\textstyle 
|S| = |S_r|+|S_{0,1}|+|S_{1,0}|\leq \frac{3}{2}|\Omega|.
\]
Moreover, by the above arguments it follows that $|S|=\frac{3}{2}|\Omega|$ if and only if $n_s=1$ for all $s\in S$.
\end{proof}
Since the maximum rank-order ratio is attained by nonregular thin Jordan schemes, our next goal is to classify them. This will be proceeded in the next subsection.

\subsection{Nonregular thin Jordan schemes}\label{nonreg}
Let us begin by describing how every commutative homogeneous coherent algebra gives rise to a nonregular homogeneous coherent Jordan algebra.
\begin{definition}	
   Let $\mathcal{C}\le M_n(\mathbb{F})$ be a commutative homogeneous coherent algebra. We define $\cJ(\mathcal{C})\subseteq M_{2n}(\mathbb{F})$ as follows: 
   \[
      \cJ(\mathcal{C}) := \left\{ \begin{bsmallmatrix} A & B \\[2mm] C & A\end{bsmallmatrix} \Mid A,B,C\in \mathcal{C}\right\}.
   \]
\end{definition}
	
\begin{prop}\label{clm1}
   For every commutative homogeneous coherent algebra $\mathcal{C}\le M_n(\mathbb{F})$ we have that $\cJ(\mathcal{C})$ is a homogeneous coherent Jordan algebra. 
\end{prop}
	\begin{proof}
		Clearly, $\cJ(\mathcal{C})$ is a linear subspace of $M_{2n}(\mathbb{F})$. 
		Since $\mathcal{C}$ is itself closed under Schur-Hadamard multiplication, so is $\cJ(\mathcal{C})$. 
        
        It remains to show closure of the Jordan product. To this end, let $A,B,C,D,E,F\in \mathcal{C}$. Then 
		\[
		2\begin{bmatrix} A & B \\[1mm] C & A\end{bmatrix}\star\begin{bmatrix}D & E \\[1mm] F & D\end{bmatrix} = \begin{bmatrix} 2AD+BF+EC & 2AE+2BD\\[2mm] 2DC + 2AF & 2AD+BF+EC\end{bmatrix}\in\cJ(\mathcal{C}).\qedhere
		\]
	\end{proof}
	
	An important special case is when $\mathcal{C}$ is the adjacency algebra of a thin association scheme coming from an abelian group $G$. In this case, we shall write $\cJ(G)$ rather than the more accurate, albeit cumbersome,
     $\cJ(\mathcal{A}(G,G))$. Clearly, for every abelian group $G$ we have that $\cJ(G)$ is a nonregular thin homogeneous coherent Jordan algebra. 
	\begin{prop}\label{clm2}
		Let $\mathfrak{S}=(\Omega,\mathcal{C})$ be a nonregular thin Jordan scheme. and let $\mathcal{A}$ be its adjacency algebra. Then there exists an abelian group $G$ and a permutation matrix $P$, such that $\mathcal{A} = P^{-1}\cJ(G)P$. 
	\end{prop}
	\begin{proof}
		Let $\mathfrak{J}=(\Omega,\mathcal{C})$ be a nonregular thin Jordan scheme and let $\mathcal{A}$ be its adjacency algebra. Let $\cS=\{A_i\}_{i=1}^r$ be the standard basis of $\mathcal{A}$. By \cite[Proposition 2.2]{MuzReiKli22}, each element of $\cS$ is of one of the following shapes:
		\[
		\begin{bmatrix}
			A & O \\[1mm]
			O & B
		\end{bmatrix}, \quad
		\begin{bmatrix}
			O & C \\[1mm] 
            O & O
		\end{bmatrix},\quad
		\begin{bmatrix}
			O & O \\[1mm]
			D & O
		\end{bmatrix},
		\] 
		where $A,B,C,D$ are certain matrices of order $n$, and where $O$ denotes the zero matrix. Since $\mathfrak{S}$ is thin, each of $A,B,C,D$ is a permutation matrix.
        
		Let us choose the following designations:   
        \begin{align*}
			\cG_1 &= \left\{A\Mid\exists B\,:\, 
			\begin{bsmallmatrix}
				A & O \\[1mm] O & B 	
			\end{bsmallmatrix}\in \cS\right\},&
			\cG_2 &= \left\{B\Mid\exists A\,:\, 
			\begin{bsmallmatrix}
				A & O \\[1mm] O & B 	
			\end{bsmallmatrix}\in \cS\right\},\\
			\cH &= \left\{C\Mid 
			\begin{bsmallmatrix}
				O & C \\[1mm] O & O 	
			\end{bsmallmatrix}\in \cS\right\},&
			\cK &= \left\{D\Mid 
			\begin{bsmallmatrix}
				O & O \\[1mm] D & O 	
			\end{bsmallmatrix}\in \cS\right\}.
		\end{align*}
		Note that $\cK=\cH^t=\cH^{-1}$ and $\sum \cG_1 = \sum \cG_2 = \sum \cH = \sum \cK = J$. Moreover, $|\cG_1| = |\cG_2| = |\cH| = |\cK| = n$. 
		
		We next prove
		\begin{equation}\label{claim1}
			\forall_{A\in \cG_1}\,\exists!_{B\in \cG_2}\,:\, \begin{bsmallmatrix} A & O\\[2mm] O & B\end{bsmallmatrix}\in \cS.
		\end{equation}
		Suppose $\begin{bsmallmatrix} A & O\\[1mm]O & B_1\end{bsmallmatrix}, \begin{bsmallmatrix} A & O\\[1mm]O & B_2\end{bsmallmatrix}\in \cS$ with $B_1\neq B_2$. Then $\begin{bsmallmatrix} O & O \\[1mm] O & B_1-B_2\end{bsmallmatrix}\in \mathcal{A}$, whence also 
		\[
		\begin{bmatrix} 
			O & O\\[1mm] 
			O & B_1 - B_2	
		\end{bmatrix}\circ
		\begin{bmatrix} 
			A & O\\[1mm] 
			O & B_1 	
		\end{bmatrix}=	
		\begin{bmatrix} 
			O & O\\[1mm] 
			O & B_1 	
		\end{bmatrix}\in \cA.
		\]
		But since $\mathcal{A}$ is closed under taking powers, we thereby obtain $\begin{bsmallmatrix} O & O\\[1mm] O & I\end{bsmallmatrix}\in\mathcal{A}$, a contradiction under the assumption that $\mathfrak{J}$ is a Jordan-scheme. It therefore
        follows that $B_1=B_2$.
		
		This establishes a bijection between $\cG_1$ and $\cG_2$ which we denote by $A\mapsto \varphi(A)$. 
		Clearly, $\cG_1$ is the standard basis of a thin homogeneous coherent $J$-algebra, and likewise for $\cG_2$.
		
		We next show that 
		\begin{equation}\label{claim3}
			\forall_{C\in \cH}\,\forall_{A\in \cG_1}\,\exists_{D\in \cK}\,:\, \begin{bsmallmatrix} A & O\\[2mm] O & \varphi(A)\end{bsmallmatrix} = \begin{bsmallmatrix} CD & O \\[2mm] O & DC\end{bsmallmatrix}. 
		\end{equation}
		To see this, first observe that for all $C\in \cH$, $D\in \cK$ we have
		\[
		2 \begin{bmatrix} O & O \\[1mm] D & O \end{bmatrix}\star\begin{bmatrix}O & C\\[1mm] O & O\end{bmatrix} = \begin{bmatrix} CD & O \\[1mm] O & DC\end{bmatrix}.
		\]
		Since $|C\cK| = |\cK|=|\cG_1|$, there exists $D\in \cK$ for which $A = CD$. By \eqref{claim1}, $\varphi(A) = DC$ holds for this $D$.
		
		Next we establish 
		\begin{equation}\label{claim4}
			\forall_{C\in \cH}\,\forall_{A\in \cG_1}\,:\, \varphi(A) = C^{-1} A C.	
		\end{equation}
		To this end, let $A\in \cG_1$. By \eqref{claim3}, there exists $D\in \cK$ such that $A=CD$ and $\varphi(A) = DC$. From this it follows that $D=C^{-1}A$, hence $\varphi(A) = C^{-1} A C$. 
		
		Let $P$ be a permutation matrix and let $X\in \cA$ be arbitrary. Then conjugation $X\mapsto PXP^{-1}$ preserves each of the operations $\cdot$, $\tau$, $\star$, $\circ$ on $\cA$. In other words, the resulting algebra $\mathcal{A}'= P\mathcal{A}P^{-1}$ is a homogeneous coherent $J$-algebra, and the Jordan schemes corresponding to $\mathcal{A}$ and $\mathcal{A'}$ are combinatorially isomorphic. 	
		
		Next we claim that:
		\begin{equation}\label{claim6}
			\exists_P\,:\, \text{$P$ is a permutation matrix, and  }\forall_{A\in \cG_1}\,:\,  P\begin{bmatrix} A & O \\[1mm] O & \varphi(A)\end{bmatrix} P^{-1}= \begin{bmatrix} A & O \\[1mm] O & A \end{bmatrix}.
		\end{equation}
		To see this, let $C\in \cH$ and let $P:= \begin{bsmallmatrix} I & O \\[1mm] O & C\end{bsmallmatrix}$. Then from \eqref{claim4} it follows that 
		\[
		P \begin{bmatrix} A & O \\[1mm] O & \varphi(A)\end{bmatrix} P^{-1} = \begin{bmatrix} I & O \\[1mm] O & C \end{bmatrix} \begin{bmatrix} A & O \\[1mm] O & C^{-1} A C \end{bmatrix} \begin{bmatrix} I & O \\[1mm] O & C^{-1} \end{bmatrix} = \begin{bmatrix} A & O \\[1mm] O & A \end{bmatrix}.
		\]
		
		By \eqref{claim6}  we may now assume that $\cG_1=\cG_2$ and moreover that $\varphi$ is the identity-map. Henceforth, we make these assumptions.
		
		Since $A = \varphi(A) = C^{-1} A C$, an immediate consequence of \eqref{claim4} is  
		\begin{equation}\label{claim5}
			\forall_{A\in \cG_1},\,\forall_{C\in \cH}\,:\, AC=CA.	
		\end{equation}
		 
		For $A\in \cG_1$ and $C\in \cH$,  we next compute
		\[	
        \begin{bmatrix} A & O \\[1mm] O & A \end{bmatrix}
		\star\begin{bmatrix} O & C\\[1mm] O & O\end{bmatrix} 
		= \half\begin{bmatrix} O & AC + CA \\[1mm] O & O\end{bmatrix} 
		= \begin{bmatrix} O & AC\\[1mm] O & O\end{bmatrix}
		\]
which proves 
        \begin{equation*}
			\forall_{A\in \cG_1},\,\forall_{C\in \cH}\,:\, AC\in \cH.	
		\end{equation*}

		We next claim $\cG_1$ is an abelian group. 	
		To this end, let $C\in \cH$ and let $A,A'\in \cG_1$. Then
        \begin{equation*}
		 AA'C = A(A'C) = (A'C)A = A'AC\implies 
         AA'=A'A\implies
         A\star A' = AA',
         \end{equation*}
		and the claim is proved. 
        
		Finally, it remains to show that $\cH=\cK=\cG_1$.	
		For this purpose, let $C=(c_{i,j})$ be the unique element of $\cH$ with $c_{1,1}=1$, i.e.\  $C\vec{e}_1 = \vec{e}_1$. (Note that this element exists because $\sum \cH = J$.) Let $k\in \{1,\dots,n\}$ and let $A_k=\big(a^{(k)}_{i,j}\big)$ be the unique element of $\cG_1$ with $a^{(k)}_{1,k}=1$, i.e.\ $A_k\vec{e}_1=\vec{e_k}$. (This element exists due to $\sum \cG_1 = J$.) 
		
		Since $A_k C\colon \vec{e}_1\mapsto \vec{e}_k$,  $CA_k\colon \vec{e}_1\mapsto C\vec{e}_k$, and $A_kC=CA_k$ by (\ref{claim5}), it follows that $C = I$. We thereby conclude that $\cH=\{AC\mid A\in \cG_1\}= \cG_1$. 
		As $\cK=\cH^{-1}$, we also have $\cK=\cG_1$.  
		But now it is clear that $\mathcal{A} = \cJ(\cG_1)$. 
	\end{proof}
Let us collect our findings in one place:
\begin{theorem} ~\\[-3ex]
\begin{enumerate}[(a)]
    \item For every finite abelian group $G$, $\cJ(G)$ is the adjacency algebra of a non-regular thin Jordan scheme.
    \item For every non-regular thin Jordan scheme $\mathfrak{S}$ with adjacency algebra $\cA$, there exists a finite abelian group $G$ such that $\cA\cong \cJ(G)$. 
    \item For finite abelian groups $G$ and $H$, one has $\cJ(G)\cong\cJ(H)$ if and only if $G\cong H$. 
    \end{enumerate}
\end{theorem}
\begin{proof}
    Parts (a) and (b) follow from Propositions~\ref{clm1} and \ref{clm2}, respectively. 
    For part (c), clearly, if $G\cong H$ then $\cJ(G)\cong \cJ(H)$. For the reverse direction we note that $G$ can be reconstructed from $\cJ(G)$ as follows: observe that the WL-closure $\mathcal{W}$ of $\cJ(G)$ is equal to 
    \[
    \mathcal{W} = \left\{ \begin{bsmallmatrix} A & B \\[2mm] C & D\end{bsmallmatrix} \Mid A,B,C,D\in \cA(G,G)\right\}.
    \]
    In other words it is a coherent algebra with two fibres where each fibre constituent coincides with $\cA(G,G)$. 
\end{proof}
    \subsection{Regular thin Jordan  schemes}
	
    In this subsection, we assume that $(\Omega,S)$ is a regular thin \JS, that is, $S$ is a set of permutations on $\Omega$. The set $S$ contains the identity permutation $1_\Omega$ and is closed under transpose, i.e.\ $s^{t}\in S$ for every $s\in S$. In fact  $s^t=s^{-1}$.
	
	Since $S$ is a partition of $\Omega^2$ into a disjoint union of permutations, for every pair $(\alpha,\beta)\in\Omega^2$ the basic relation $S(\alpha,\beta)\in S$ is the unique permutation that maps $\alpha$ to $\beta$. From the definition of a \JS it follows that $S(\alpha,\alpha)=1_\Omega$ and $S(\alpha,\beta)
	=S(\beta,\alpha)^{-1}$. The same definition implies that the adjacency matrix $(S(\alpha,\beta))_{\alpha,\beta\in\Omega}$ is a \emph{Latin} square. Pick an arbitrary point $\omega_0\in\Omega$ and define a binary operation $\diamond$ on $S$ by $a\diamond b = S(\omega_0,a(b(\omega_0)))$. That is, a triple $a,b,c\in S$ satisfies $a\diamond b=c$ iff $a(b(\omega_0)) = c(\omega_0)$.
	
	\begin{remark} It follows from the definition that the product $\diamond$ depends on the  choice $\omega_0$ of ``base point". It could be that a different choice of  base point will yield a nonisomorphic magma. 
	\end{remark}
	
	\begin{prop}\label{loop}
		The magma $(S,\diamond)$ has the following properties
		\begin{enumerate}[(a)]
			\item \label{loop_a} $1_\Omega$ is the neutral element of $(S,\diamond)$;
			\item \label{loop_b} $\forall_{a\in S}\ a^{t}\diamond a =1_\Omega$;
			\item \label{loop_c} $\forall_{a,b\in S}\ a^{t}\diamond (a\diamond b)=b$;
			\item \label{loop_d} $\forall_{a,b\in S}$ the equation $x\diamond a = b$ has a unique solution $x\in S$.
		\end{enumerate}	
	\end{prop} 
	\begin{proof}
{\sl Proof of \ref{loop_a}:} We have $a\diamond 1_\Omega = S(\omega_0,a(1_\Omega(\omega_0)))= S(\omega_0,a(\omega_0)) = a$. Similarly, $1_\Omega\diamond a = a$.
\smallskip
		
{\sl Proof of \ref{loop_b}:} This follows at once from the identity $a^{t}(a(\omega_0))=\omega_0=1_\Omega(\omega_0)$.
\smallskip
		
{\sl Proof of \ref{loop_c}:}
		$a\diamond  b=c\Leftrightarrow a(b(\omega_0)) = c(\omega_0)\Leftrightarrow b(\omega_0)=a^{t}(c(\omega_0))\Leftrightarrow a^{t}\diamond c=b$.
        \smallskip
		
{\sl Proof of \ref{loop_d}:} We have $x(a(\omega_0))=b(\omega_0)$. Therefore, $x=S(a(\omega_0),b(\omega_0))$.
	\end{proof}
	
	We conclude from the above that $(S,\diamond)$ is a loop with the left inverse property. It is well known that each element $a\in S$  determines two permutations $\ell_a$ and $r_a$ on $S$, defined naturally by $\ell_a(x)=a\diamond x$ and $r_a(x)=x\diamond a$.  From \ref{loop_c} it follows that $\ell_{a^{t}} \ell_a =1_S$, that is, $\ell_{a^t}=\ell_a^{-1}$. It also follows from \ref{loop_c} that left division in the loop $(S,\diamond)$ is $a^{t}\diamond b$. Right division will be denoted as $b/a$, that is, $(b/a)\diamond a = b$. Later we will prove $b/a = b\diamond a^{t}$.

Recall that an 
\emph{isotopism} between two Latin squares $L:\Omega^2\rightarrow S$ and $L':{\Omega'}^2\rightarrow S'$ is a triple of  bijections $\alpha\!:\!\Omega\rightarrow \Omega'$, $\beta\!:\!\Omega\rightarrow \Omega'$ and $\gamma\!\!:S\rightarrow S'$ that satisfy the identity $\gamma(L(\omega_1,\omega_2)) = L'(\alpha(\omega_1),\beta(\omega_2))$ for  all $\omega_1,\omega_2\in \Omega$.
    
	\begin{prop}\label{conj} Let $\sigma:S\rightarrow \Omega$ be the mapping defined via $\sigma(s)=s(\omega_0)$. Then 
		\begin{enumerate}[(a)]
            \item \label{conj_a} $\sigma$ is a bijection.
			\item \label{conj_b} $\sigma^{-1} a \sigma = \ell_a$  for each $a\in S$.
			\item \label{conj_c} The triple of mappings $\sigma,\sigma,1_S$ is an isotopism between the Latin squares $(v/u)_{v,u\in S}$ and $(S(\alpha,\beta))_{\alpha,\beta\in\Omega}$.
		\end{enumerate}
	\end{prop}
	\begin{proof}
        {\sl Proof of \ref{conj_a}:} If $s\neq t\in S$, then $s\cap t=\emptyset$. Thus $\sigma(s)=s(\omega_0)\neq t(\omega_0)=\sigma(t)$ which proves $\sigma$ is injective. To prove surjectivity, pick an arbitrary $\omega\in\Omega$ and set $s:=S(\omega_0,\omega)$. Then $\sigma(s)=s(\omega_0)=\omega$.
        \smallskip
        
		{\sl Proof of \ref{conj_b}:} We must show $a\sigma = \sigma \ell_a$ for each  $a\in S$. But this is equivalent to showing  $\forall_{s\in S}:\ a(\sigma(s)) =\sigma(\ell_a(s))$. Now $a(\sigma(s)) = a(s(\omega_0)) = (a\diamond s)(\omega_0) = \sigma(\ell_a(s))$.
        \smallskip
		
		{\sl Proof of \ref{conj_c}:} We must show $S(\sigma(u),\sigma(v)) = 1_S(v/u)=v/u$ for all $u,v\in S$. Denote $w:=v/u$.
		Then $w\diamond u = v$ which implies $w(u(\omega_0))=v(\omega_0)$. Hence $w = S(\sigma(u),\sigma(v))$.
	\end{proof}
	Proposition \ref{conj}  implies that the permutations $\{\ell_s\}_{s\in S}$ form a Jordan scheme on $S$ isomorphic to the original  $(\Omega,S)$. Therefore $\forall_{u,v\in S}\  \und{\ell_u}\cdot\und{\ell_v} + \und{\ell_v}\cdot\und{\ell_u} = \underline{\ell_{u\diamond v}} + \underline{\ell_{v\diamond u}}$, or equivalently,
    \begin{equation}\label{jordan}
 \forall_{u,v\in S}\ \ell_u\ell_v\cup \ell_v\ell_u = \ell_{u\diamond v} \cup \ell_{v\diamond u}.
	\end{equation} 
	This latter equality may be re-expressed as  
	\begin{equation}\label{c-assoc}
		\forall_{u,v,w\in S}\ \{u\diamond (v\diamond w), v\diamond (u\diamond w)\} = \{(u\diamond v)\diamond w, (v\diamond u)\diamond w)\}.
	\end{equation}
	As an immediate consequence we conclude that $(S,\diamond)$ has the left alternative property: $u\diamond (u\diamond v) = (u\diamond u)\diamond v$.
	
	For each pair $u,v\in S$, we set ${\mathfrak a}({u,v}):=\{w\in S\mid u\diamond (v\diamond w) =(u\diamond v)\diamond w\}$. Observe that $1_\Omega\in \mathfrak a({u,v})$. It follows from~\eqref{c-assoc} that $\mathfrak a({u,v})=\mathfrak a({v,u})$. 
	
	\subsubsection{$(S,\diamond)$ is a Moufang loop}
	
	To show this we start with the following assertion:
	
	\begin{prop}\label{commute} Let $u,v\in S$ be an arbitrary pair. Then the following hold:
		\begin{enumerate}[(a)]
			\item If $u\diamond v = v\diamond u$, then $\ell_u\ell_v=\ell_v\ell_u=\ell_{u\diamond v}$ and $\mathfrak a({u,v}) = S$.
			\item If $u\diamond v \neq v\diamond u$, then $\ell_u\ell_v\cap \ell_v\ell_u=\emptyset$. 
		\end{enumerate}
	\end{prop}
	\begin{proof}
		{\sl Proof of \ref{conj_a}:} This is an immediate consequence of \eqref{c-assoc}.
        \smallskip
		
		{\sl Proof of \ref{conj_b}:} As $u\diamond v\neq v\diamond u$, we have  $\ell_{u\diamond v}\cap \ell_{v\diamond u} = \emptyset$. But then 
		$(u\diamond v)\diamond w\neq (v\diamond u)\diamond w$ holds for every $w\in S$, whence the result follows once more from  \eqref{c-assoc}. 
	\end{proof}
	
	\begin{prop}\label{part} 
		Let $u,v\in S$, and let $\mathfrak a:=\mathfrak a({u,v})$ and $\overline{\mathfrak a}:=S\setminus \mathfrak a$. Then the following hold:
        \begin{enumerate}[(a)]
            \item $\ell_{u\diamond v} = \ell_u \ell_v 1_{\mathfrak a} \cup \ell_v \ell_u 1_{\overline{\mathfrak a}}$; in particular, $\ell_{u\diamond v} \subseteq \ell_u \ell_v \cup \ell_v \ell_u$.
            \item The set $\mathfrak a$ is invariant with respect to  $\ell_u$ and $\ell_v$.
        \end{enumerate} 
	\end{prop} 
	
	\begin{proof} 
    If $u$ and $v$ commute, then we are done by Proposition~\ref{commute}(a). So for the balance of the proof we assume $u\diamond v\ne v\diamond u$. 
    
    {\sl Proof of \ref{conj_a}:} By~\eqref{jordan}, $\ell_u \ell_v\cup \ell_v \ell_u = \ell_{u\diamond v}\cup \ell_{v\diamond u}$ while by~\eqref{c-assoc},  $(u\diamond v)\diamond w \in \{u\diamond (v\diamond w), v\diamond (u\diamond w) \}$. More precisely, 
		\[
		(u\diamond v)\diamond w = \left\{
		\begin{array}{rl}
			u\diamond (v\diamond w), & w \in \mathfrak a,\\
			v\diamond (u\diamond w), & w \in \overline{\mathfrak a}.
		\end{array}
		\right.
		\]
		But this is equivalent to $\ell_{u\diamond v} = \ell_u \ell_v 1_{\mathfrak a}\cup \ell_v \ell_u 1_{\overline{\mathfrak a}}$. 
        \smallskip
        
       {\sl Proof of \ref{conj_b}:} From (a) and the fact that
		 $u^t\diamond (u\diamond v) = v$, it follows that $\ell_v\subseteq \ell_{u^t}\ell_{u\diamond v}\cup \ell_{u\diamond v}\ell_{u^t}$. Taking into account that $\ell_{u^t}=\ell_u^{-1}$ we now obtain
		\[
		\ell_v\subseteq \ell_u^{-1}(\ell_u \ell_v 1_{\mathfrak a}\cup \ell_v \ell_u 1_{\overline{\mathfrak a}})\cup (\ell_u \ell_v 1_{\mathfrak a}\cup \ell_v \ell_u 1_{\overline{\mathfrak a}})\ell_u^{-1}
		\]
		which  simplifies to 	                  
        \begin{equation}\label{231124a}
			\ell_v\subseteq 
			\ell_v 1_{\mathfrak a} \cup \ell_u^{-1} \ell_v \ell_u 1_{\overline{\mathfrak a}} \cup \ell_u \ell_v 1_{\mathfrak a}\ell_u^{-1}\cup \ell_v \ell_u 1_{\overline{\mathfrak a}}\ell_u^{-1}.
		\end{equation}

		Below we determine the cardinality of the intersection of $\ell_v$ with each set appearing in~\eqref{231124a}.
		\[
		\begin{array}{ll}
			|\ell_v 1_{\mathfrak a}\cap \ell_v|=|\mathfrak a| & \ \\
			|\ell_u^{-1} \ell_v \ell_u 1_{\overline{\mathfrak a}}\cap \ell_v|\leq |\ell_u^{-1} \ell_v \ell_u \cap \ell_v| = |\ell_v \ell_u\cap \ell_u \ell_v|=0  \Rightarrow & |\ell_u^{-1} \ell_v \ell_u 1_{\overline{\mathfrak a}}\cap \ell_v| = 0\\
			|\ell_u \ell_v 1_{\mathfrak a}\ell_u^{-1}\cap \ell_v|\leq |\ell_u \ell_v \ell_u^{-1}\cap \ell_v|=|\ell_u \ell_v\cap \ell_v \ell_u|=0  \Rightarrow &
			|\ell_u \ell_v 1_{\mathfrak a} \ell_u^{-1}\cap \ell_v|=0\\
			|\ell_v \ell_u 1_{\overline{\mathfrak a}}\ell_u^{-1}\cap \ell_v|=|	\ell_v \ell_u 1_{\overline{\mathfrak a}}\cap \ell_v \ell_u| = |\overline{\mathfrak a}| & \ 
		\end{array}
        \] 
		Together, this  implies $(\ell_v 1_{\mathfrak a}\cup \ell_v \ell_u 1_{\overline{\mathfrak a}}\ell_u^{-1})\cap \ell_v =\ell_v$ which in turn yields $\ell_u 1_{\overline{\mathfrak a}}\ell_u^{-1}=1_{\overline{\mathfrak a}}$,  or equivalently, $\ell_u(\overline{\mathfrak a})=\overline{\mathfrak a}$. Similarly, we obtain the equality $\ell_v(\overline{\mathfrak a})=\overline{\mathfrak a}$ from  $u\leftrightarrow v$ symmetry.
	\end{proof}
	It follows from Proposition~\ref{part}(b) that $u\diamond \mathfrak a({u,v})= \mathfrak a({u,v})$ and $v\diamond \mathfrak a({u,v})=\mathfrak a({u,v})$. Together with $1_\Omega\in \mathfrak a({u,v})$, we see that the subloop of $(S,\diamond)$ generated by $u$ and $v$ is contained in $\mathfrak a({u,v})$. In particular, $u,v\in \mathfrak a({u,v})$. 
	
	\begin{prop}\label{inverse}
		For all $u,v\in S$ one has $(v\diamond u)^t = u^t\diamond v^t$.
	\end{prop}
	\begin{proof} Since  $u\diamond \mathfrak a({u,v})= \mathfrak a({u,v})$, it follows that $u^t\diamond(u\diamond \mathfrak a({u,v}))=u^t\diamond \mathfrak a({u,v})$. By Proposition~\ref{loop}\ref{loop_c}, $u^t\diamond(u\diamond \mathfrak a({u,v})) = \mathfrak a({u,v})$ whence $u^t\diamond \mathfrak a(u,v) = \mathfrak a({u,v})$. Analogously, $v^t\diamond \mathfrak a(u,v) = \mathfrak a({u,v})$.  Now since  $1_\Omega\in \mathfrak a(u,v)$, we obtain $u^t\diamond v^t\in \mathfrak a({u,v})=\mathfrak a({v,u})$. Therefore $(v\diamond u)\diamond (u^t\diamond v^t) = 
		v\diamond (u\diamond (u^t\diamond v^t)) = v\diamond v^t =1_\Omega$. Together with $(v\diamond u)\diamond (v\diamond u)^t=1_\Omega$, we conclude that $(v\diamond u)^t = u^t\diamond v^t$ as desired.
	\end{proof}
	
	\begin{prop}
    The loop $(S,\diamond)$ is alternative, that is to say, for all $u,v\in S$ we have 
		\[
		\begin{array}{rclc}
			(u\diamond u)\diamond v & = & u\diamond (u\diamond v) & (L)\\
			(v \diamond u)\diamond u & = & v\diamond (u\diamond u) & (R)\\
			(u \diamond v)\diamond u & = & u\diamond (v\diamond u) & (M)\\
		\end{array}
		\]
	\end{prop}
	\begin{proof}
		Equality (L) follows from equation~\eqref{c-assoc}, while  (R) and (M) follow from the inclusion $u\in \mathfrak a({u,v})=\mathfrak a({v,u})$. 		
	\end{proof}
	
	\begin{prop}\label{moufang}  $(S,\diamond)$ is a Moufang loop.
	\end{prop}
	\begin{proof}
		A loop is Moufang iff it is simultaneously a left and right Bol loop. 
		We start by showing that $(S,\diamond)$ is a left Bol loop. To this  end, we must prove the identity 
        \begin{equation}\label{lbol}
        y\diamond (z\diamond(y\diamond x))=(y\diamond(z\diamond y))\diamond x. 
        \end{equation}
        But this identity is equivalent to $\ell_y \ell_z \ell_y = \ell_{y\diamond(z\diamond y)}$. To establish the latter identity, we first note that $\cJ:=\langle \und{\ell}_s\rangle_{s\in S}$ is $\star$-closed. Therefore $\und{\ell}_y\und{\ell}_z\und{\ell}_y = 2(\und{\ell}_y\star(\und{\ell}_z\star\und{\ell}_y) - (\und{\ell}_y\star \und{\ell}_y)\star \und{\ell}_z)\in \cJ$. This implies the permutation matrix $\und{\ell}_y\und{\ell}_z\und{\ell}_y$ is a linear combination of pairwise disjoint permutation matrices $\und{\ell}_s,s\in S$. But this can only occur if $\und{\ell}_y\und{\ell}_z\und{\ell}_y =\und{\ell}_w$ for some $w\in S$, or equivalently, $\ell_y\ell_z\ell_y=\ell_w$. We further argue that since $(S,\{\ell_s\}_{s\in S})$ is a Jordan scheme, we have $\ell_w = \ell_{y\diamond(z\diamond y)}$ if and only if $\ell_w(1_S) = \ell_{y\diamond(z\diamond y)}(1_S)$. 
        (We have here used the fact that in a Jordan scheme two relations  are either equal or disjoint.)  We now compute
        \[
        \ell_w(1_S) = \ell_y(\ell_z(\ell_y(1_S))) = y\diamond(z\diamond y) = \ell_{y\diamond(z\diamond y)}(1_S).
        \]
        Thus identity~\eqref{lbol} holds, and $(S,\diamond)$ is indeed a left Bol loop.

        To see that $(S,\diamond)$ is a right Bol loop, we must establish the  identity
        \begin{equation}\label{rbol}
            ((x\diamond y)\diamond z)\diamond y = x\diamond ((y\diamond z)\diamond y).
        \end{equation}
        Similarly to the above, this identity is equivalent to  $r_y r_z r_y = r_{(y\diamond z)\diamond y}$. From Proposition~\ref{inverse} it follows that $r_z = \tau \ell_{z^t} \tau$ where $\tau\in\mathsf{Sym}(S)$ is the involutory transpose permutation on $S$, i.e.\ $\tau(x)=x^t$. Since $(S,\diamond)$ is a left Bol loop, we have $\ell_{y^t}\ell_{z^t}\ell_{y^t} = \ell_{y^t\diamond(z^t\diamond y^t)}$.  Conjugating both sides of this identity by $\tau$, we obtain on one hand
        \[
        \tau \ell_{y^t}\ell_{z^t}\ell_{y^t} \tau = (\tau\ell_{y^t}\tau)(\tau\ell_{z^t}\tau)(\tau\ell_{y^t}\tau) = r_y r_z r_y,
        \]
        while on the other hand
        \[
        \tau \ell_{y^t\diamond(z^t\diamond y^t)}\tau = \tau \ell_{((y\diamond z)\diamond y)^t}\tau = r_{(y\diamond z)\diamond y}.
        \]
     Thus identity~\eqref{rbol} holds, and $(S,\diamond)$ is a right Bol loop as well.%
	\end{proof}
	
	\begin{prop}
		Let $(\Omega, S)$ be a thin symmetric Jordan scheme.  Then $S$ is an abelian $2$-group and $(\Omega, S)$ is a symmetric association scheme.
	\end{prop}
	\begin{proof} Since $a^t=a$, each element of the loop $(S,\diamond)$ has order two. It is well known that a Moufang loop of exponent two is a group (cf.\ \cite[Proposition 2]{Che74}). Hence $(S,\diamond)$ is an elementary abelian $2$-group.
	\end{proof}
	
	\subsubsection{Relation to alternative algebras}
	
	A loop $(S,\diamond)$ is \emph{ring-alternative} (an RA-loop, for short \cite{CheGoo86}) if the loop ring $RS$ is an alternative ring for any commutative associative ring  $R$. It is known that every RA-loop is Moufang but the converse is not true \cite{CheGoo86}. RA-loops were characterized in \cite[Theorem 1]{CheGoo86} as follows: 
	\begin{theorem}\label{RA} A nonassociative loop $(S,\diamond)$ is an RA-loop if and only if the following hold:
		\begin{enumerate}[(a)]
			\item If three elements in $S$ associate in some order, then they associate in all orders.
			\item If three elements $u,v,w$ do not associate, then  $(u\diamond v)\diamond w=u\diamond (w\diamond v)=v\diamond (u\diamond w)$.
		\end{enumerate}
	\end{theorem}
	\begin{theorem}
        Let $(S,\diamond)$ be a loop. The permutations $\{\ell_a\}_{a\in S}$ form a thin Jordan scheme if and only if $(S,\diamond)$ is an RA-loop.
	\end{theorem}
	\begin{proof}
		($\Leftarrow$) Let $(S,\diamond)$ be an RA-loop. Then $(S,\diamond)$ is a Moufang loop \cite{CheGoo86}, and therefore every element $a$ has an inverse $a^{-1}$ satisfying $\forall_{a,b\in S}\  a\diamond (a^{-1}\diamond b) = b$. This implies that $\ell_{a^{-1}}= \ell_a^{-1} = \ell_a^{\,t}$. Together with $\ell_{1_\Omega}=1_S$, we conclude that the partition $L:=\{\ell_a\}_{a\in S}$ of $S^2$ is a  rainbow. If $\diamond$ is associative, then $(L,S)$ is a thin \AS, and therefore it is a Jordan scheme as well. 
        
        If $\diamond$ is nonassociative, then $(L,S)$ is a Jordan scheme if and only if  
        \[\forall_{a,b\in S}:\,\und{\ell_a}\cdot\und{\ell_b}+\und{\ell_b}\cdot\und{\ell_a} = \und{\ell_{a \diamond b}} + \und{\ell_{b\diamond a}}.
        \] 
        Observe that this is equivalent to the following relation in the loop ring $RS$: \begin{equation}\label{abc}
			a\diamond (b\diamond c) + b\diamond (a\diamond c) = (a\diamond b)\diamond c + (b\diamond a)\diamond c.
		\end{equation} 
		If $a,b,c$ associate in lexicographic order, then by Theorem~\ref{RA} they  associate in any order and~\eqref{abc} follows.  
		If $a,b,c$ do not associate, then by Theorem~\ref{RA} 
		$
		a\diamond (b\diamond c) = (b\diamond a)\diamond c, b\diamond (a\diamond c) = (a\diamond b)\diamond c
		$
		and~\eqref{abc} holds once more.
		
		($\Rightarrow$) Assume now that the permutations $L=\{\ell_u\}_{u\in S}$ form a Jordan scheme. Choosing $\omega_0=1_\Omega$ we obtain a binary operation $\odot$ on $L$ via 
        \[\ell_a\odot \ell_b =\ell_c \iff \ell_c(\omega_0) = \ell_a(\ell_b(\omega_0)) \iff c=a\diamond b.
        \]
        Thus it always holds that  $\ell_a\odot \ell_b = \ell_{a\diamond b}$. Therefore 
        the mapping $a\mapsto \ell_a$ is an isomorphism between the magmas $(L,\odot)$ and $(S,\diamond)$. 
        
        By Proposition~\ref{moufang}, $(L,\odot)$ is a Moufang loop. We must check the two conditions of Theorem~\ref{RA}. If $\ell_u,\ell_v,\ell_w$ associate, then they generate a subgroup of $L$ and therefore associate in any order. (Here we have used the fact  that any subloop of a  Moufang loop that is generated by three associate elements is associative.)  Thus the first condition of Theorem~\ref{RA} is satisfied. 
		
		To check the second condition, we  consider a nonassociate triple  $\ell_u,\ell_v,\ell_w$, i.e.\ $\ell_u\odot(\ell_v\odot \ell_w)\neq (\ell_u\odot\ell_v)\odot \ell_w$. Then it follows from~\eqref{c-assoc} that $\ell_u\odot (\ell_v\odot \ell_w) = (\ell_v\odot \ell_u)\odot \ell_w$ and 
		$\ell_v\odot (\ell_u\odot \ell_w) = (\ell_u\odot \ell_v)\odot \ell_w$. 
		
		It remains to show that $(\ell_u\odot \ell_v)\odot \ell_w = \ell_u\odot (\ell_w\odot \ell_v)$. 
		Since $\ell_u,\ell_v,\ell_w$ do not associate, the elements $\ell_w^{-1},\ell_v^{-1},\ell_u^{-1}$ do not associate either. Now~\eqref{c-assoc} implies $\ell_w^{-1}\odot (\ell_v^{-1}\odot \ell_u^{-1}) = (\ell_v^{-1}\odot \ell_w^{-1})\odot \ell_u^{-1}$.
		Taking the inverse of both  sides, we obtain 
		$(\ell_u\odot \ell_v)\odot \ell_w = \ell_u\odot (\ell_w\odot \ell_v)$. Combining this with the equality of the previous paragraph, we conclude that $(L,\odot)$ satisfies the second condition of Theorem~\ref{RA}. Hence $(L,\odot)$, and therefore $(S,\diamond)$, are RA-loops.
	\end{proof}
    
    We end this subsection with an isomorphism criterion for loop-based Jordan schemes. Following \cite{GMRS16}, 
     a \emph{half-isomorphism}  between two loops $(L,\diamond)$ and $(L',\diamond')$ is a bijection $\varphi :L\rightarrow L'$ satisfying $\varphi (x\diamond y)\in\{\varphi (x)\diamond' \varphi (y),\varphi (y)\diamond' \varphi (x)\}$. Note that every loop isomorphism or anti-isomorphism is a half-isomorphism, although we regard these as {trivial}.
     
    \begin{prop}\label{iso}
        Given two RA-loops $(L,\diamond)$ and $(L',\diamond')$, their corresponding Jordan schemes $\cJ:=(L,\{\ell_a\}_{a\in L})$ and $\cJ':=(L',\{\ell'_a\}_{a\in L'})$ are isomorphic (as Jordan schemes) if and only if $L$ and $L'$ are half-isomorphic. 
    \end{prop}
    \begin{proof}
        The reverse direction is obvious. To prove the forward direction, we 
         assume that $\cJ$ and $\cJ'$ are isomorphic Jordan schemes (i.e.\ rainbows). In other words, there exists a bijection $f:L\rightarrow L'$ such that 
        $
        f\{\ell_a\}_{a\in L}f^{-1} = \{\ell'_a\}_{a\in L'}$. The latter is equivalent to the existence of a bijection $\ \varphi:L\rightarrow L'$ satisfying $f\ell_a f^{-1} = \ell'_{\varphi(a)}
        $ for all $a\in L$. This yields
        $$
        f(\ell_a\ell_b\cup\ell_b\ell_a)f^{-1}= \ell'_{\varphi(a)}\ell'_{\varphi(b)}\cup \ell'_{\varphi(b)}\ell'_{\varphi(a)}.
        $$
        Since both $\cJ$ and $\cJ'$ are Jordan schemes, the above equality is equivalent to
        $$
        f(\ell_{a\diamond b}\cup\ell_{b\diamond a})f^{-1} = \ell'_{\varphi(a)\diamond'\varphi(b)}\cup \ell'_{\varphi(b)\diamond '\varphi(a)}.
        $$
        Since the left-hand side of this equality is equal  to 
        $\ell'_{\varphi(a\diamond b)}\cup\ell'_{\varphi(b\diamond a)}$, we conclude that 
        $$\{\varphi(a\diamond b), \varphi(b\diamond a)\} =\{ \varphi(a)\diamond'\varphi(b),\varphi(b)\diamond'\varphi(a)\},$$
        as desired.
    \end{proof}
    Although there exist Moufang loops that are non-trivially half-isomorphic, we are not aware of any examples of isomorphic Jordan schemes based on nontrivial half-isomorphic Moufang loops.
    
	\subsubsection{RA-loops.}
	RA-loops were classified in~\cite[Theorem 1.1]{JesLeaMil95}. To describe them, we must first introduce a special construction. 
	
	Let $G$ be a finite group satisfying 
	$G/Z(G)\cong \Z_2^2$ and  $G'=\{e,s\}\subseteq Z(G)$  where $e$ is the identity element of $G$. Note that according to~\cite{JesLeaMil95}, the property 
	$G/Z(G)\cong \Z_2^2$ implies $G'=\{e,s\}\subseteq Z(G)$. Note also that in this case there exists a $2$-subgroup $P$ of $G$ and an abelian subgroup $A$ of $G$ such that $G=P\times A$ and $P/Z(P)\cong\Z_2^2$.
	
	Define an involution $g\mapsto g^*$ ($g\in G$) by 
	\[
	g^* = \left\{
	\begin{array}{ll}
		g, & g\in Z(G)\\
		gs, & g\not\in Z(G).
	\end{array} 	\right. 
	\] 
	We next fix an arbitrary element $g_0\in Z(G)$ and define a binary operation $\diamond$ on the set $G\times\{0,1\}$ as follows:
	\begin{equation}\label{prod}
		(g,x)\diamond (h,y) =\left\{
		\begin{array}{ll}
			(gh,0), & x=y=0\\
			(hg,1), & x=0,y=1\\
			(gh^*,1), & x=1,y=0\\
			(g_0h^*g,0), & x=y=1.
		\end{array}
		\right.
	\end{equation}
 As in \cite{JesLeaMil95}, we denote this algebraic structure by $L(G,*,g_0)$.

	\begin{theorem}[{\cite[Theorem 1.1]{JesLeaMil95}}]
		The algebraic structure $(L(G,*,g_0),\diamond)$ is an RA-loop. Moreover,  any RA-loop is isomorphic to $L(G,*,g_0)$ for suitable $G$ and $g_0\in Z(G)$ as described above. 
	\end{theorem} 
	The following result lists properties of $L(G,*,g_0)$ that will be critical to our later arguments. 
	\begin{prop}\label{loop-prop} Let $L=L(G,*,g_0)$ be as above. Then 
		\begin{enumerate}[(a)]
			\item The map $x\mapsto x^*$ is an anti-automorphism of $G$;
			\item $Z(L)=Z(G)\times \{0\}$;
			\item $Z(L)$ coincides with the right and left nucleus of $L$.
		\end{enumerate}   
	\end{prop}
	\begin{proof}
		{\sl Proof of \ref{conj_a}:} Define $g^*=g\alpha(g)$ where $\alpha(g)=e$ if $g\in Z(G)$ and  $\alpha(g)=s$ otherwise. We must show $(gh)^*=h^*g^*$, which is equivalent to showing $[g,h]=\alpha(g)\alpha(h)\alpha(gh)^{-1}$.
		
		It follows from the properties of a group that $g,h\in G$ do not commute if and only if the cosets $gZ(G),$ $hZ(G),$ $ghZ(G)$ are pairwise distinct. In this case $[g,h]=s=\alpha(g)=\alpha(h)=\alpha(gh)$, and the required equality holds.
		
		On the other hand, if $g,h$ commute then they belong to an abelian subgroup $N:=\langle g,h,Z(G)\rangle$ that is normal in $G$. As the restriction of $\alpha$ to $N$ is a homomorphism, we have  $\alpha(gh)^{-1}\alpha(g)\alpha(h)=e=[g,h]$, thereby completing the proof of part (a).
        \smallskip
		
		{\sl Proof of \ref{conj_b}:} For any $g\in Z(G)$ and $h\in G$, it follows from  the product formula~\eqref{prod} that $gh=hg$, $g^*=g$. Therefore,
		\[(g,0)\diamond (h,y) = 
        \left\{
		\begin{array}{ll}
			gh, & y=0\\
			hg, & y=1
		\end{array}
		\right. 
        =\left\{
		\begin{array}{ll}
			hg, & y=0\\
			hg^*, & y=1\\
		\end{array}
		\right.
		=
        (h,y)\diamond (g,0),
		\]  
        implying $(g,0)\in Z(L)$. 
        
		From this it follows that $(g,0)$, $g\in Z(G)$, commutes with every element of $L$. If $g\not\in Z(G)$, then $(g,0)$ and $(g,1)$ do not commute. Hence, neither $(g,0)$ nor $(g,1)$ lies in $Z(L)$. 
        \smallskip
		
		{\sl Proof of \ref{conj_c}:} Given that $L$ is an RA-loop, it is also a Moufang loop. Thus its left nucleus $N_\lambda(L)$ coincides with its right nucleus $N_\rho(L)$. We claim $N_\rho(L)=Z(L)$.
        
        To this end, let $f,g,h\in G$, $x,y\in\{0,1\}$ and define 
        \begin{align*}
        \mathsf{L}&:=((g,x)\diamond (h,y))\diamond (f,0),\\
        \mathsf{R}&:=(g,x)\diamond ((h,y)\diamond (f,0)).
        \end{align*}
        Using \eqref{prod}, we compute  
		\[
		\mathsf{L} = 
		\left\{
		\begin{array}{ll}
			(gh,0)\diamond (f,0), & x=y=0\\
			(hg,1)\diamond (f,0), & x=0,y=1\\
			(gh^*,1)\diamond (f,0), & x=1,y=0\\
		(g_0h^*g,0)\diamond (f,0), & x=y=1
		\end{array}
		\right. 
		= 
		\left\{
		\begin{array}{ll}
			(ghf,0), & x=y=0\\
			(hgf^*,1), & x=0,y=1\\
			(gh^*f^*,1), & x=1,y=0\\
			(g_0h^*gf,0), & x=y=1.
		\end{array}\right.
		\]
		\[
		\mathsf{R}=
		(g,x)\diamond 
        \left\{
        \begin{array}{ll}
             (hf,0),& y=0  \\
             (hf^*,1),& y=1
        \end{array} =
        \right.
		\left\{
		\begin{array}{ll}
			(ghf,0), & x=y=0\\
			(hf^*g,1), & x=0,y=1\\
			(g(hf)^*,1), & x=1,y=0\\
			(g_0(hf^*)^*g,0), & x=y=1.
		\end{array}
		\right. 
		\]
		If $f\in Z(G)$ then $f^*=f$, which together with $(hf)^*=f^*h^*$ implies 
		$\mathsf{L}=\mathsf{R}$. Therefore, $Z(G)\times\{0\}\subseteq N_\rho(L)$.
		
		If $f\not\in Z(G)$ then $f^*=fs\not\in Z(G)$, which implies $hgf^*\neq hf^*g$ for any $g\in G$ that does not commute with $f^*$. Therefore $\mathsf{L}\neq\mathsf{R}$ for every $g$ with $[g,f]\neq e$. Hence $(f,0)\not\in N_\rho(L)$ for any $f\not\in Z(G)$.
		
		It remains to show that $(f,1)\not\in N_\rho(L)$ for any $f\in G$. For arbitrary $g,h\in G$, we compute 
		$\mathsf{L}:=((g,1)\diamond (h,1))\diamond (f,1), \mathsf{R}:=(g,1)\diamond ((h,1)\diamond (f,1))$.
		Applying~\eqref{prod}, we now obtain 
		\[
        \mathsf{L} = (g_0h^*g,0)\diamond (f,1) = (fg_0h^*g,1),\; \mathsf{R}=(g,1)\diamond (g_0f^*h,0) = (g(g_0f^*h)^*,1)=(gh^*fg_0^*,1).
        \]
		It follows from $g_0\in Z(G)$ that $g_0^*=g_0$. Thus $\mathsf{L}=\mathsf{R}$ holds iff $fh^*g=gh^*f$. If $f\not\in Z(G)$, then choosing $h=e$ and $g\not\in C_G(f)$ we obtain $\mathsf{L}\neq\mathsf{R}$. (Here $C_G(f)$ denotes the centralizer in $G$ of $f$.) If $f\in Z(G)$, then we may take $g,h\in G$ with $[h^*,g]\neq e$ to ensure $\mathsf{L}\neq\mathsf{R}$. We thereby conclude that $(f,1)\not\in N_\rho(L)$ for any $f\in G$.		
	\end{proof}

	\section{Coherent closure of thin Jordan schemes}
	
	The goal of this section is to show that the \JSs we have found in the previous section are \gen. For this purpose we compute the WL-closure $\mathfrak{C}$ of the Jordan scheme $\mathfrak{J}=(S,\{\ell_s\}_{s\in S})$ obtained from an RA-loop $(S,\diamond)$. Following, we compute the group $\jaut(\mathfrak{C})$ and use it to check that no algebraic fusion of $\mathfrak{C}$ coincides with $\mathfrak{J}$.

	\begin{prop}
        Let $\Omega$ be a finite set, and let $G\leq \sym{\Omega}$ be a transitive subgroup. Then the WL-closure of $G$ is a \CC consisting of the $2$-orbits of the group $H:=C_{\sym{\Omega}}(G)$.
	\end{prop}
	\begin{proof}
		Let  $\mathfrak{X}=(\Omega,C)$ be the coherent closure of $G$.  Since  $H$ centralizes $G$, each $h\in H$ is an automorphism of every $g\in G$. Therefore every basic relation of $\mathfrak{X}$ is a union of $2$-orbits of $H$. We must show that each $2$-orbit of $H$ is a union of basic relations of $\mathfrak{X}$.
		
		Pick an arbitrary pair $(\omega_1,\omega_2)\in \Omega^2$ and define $G_{\omega_1\rightarrow \omega_2}$ to  be the set of all $g\in G$ that map $\omega_1$ to  $\omega_2$, that is $G_{\omega_1\rightarrow \omega_2} =  \{g\in G\mid  g(\omega_1)=\omega_2\}= \{g\in G \mid (\omega_1,\omega_2)\in g\}$. Note that $G_{\omega_1\rightarrow \omega_2}$ is nonempty and $G_{\omega_1\rightarrow \omega_2} = gG_{\omega_1}$ for any $g\in G_{\omega_1\rightarrow \omega_2}$.
		
		The relation $r_{\omega_1,\omega_2} := \bigcap_{g\in G_{\omega_1\rightarrow \omega_2}} g$ is nonempty and is the union of certain elements from $C$. We claim $r_{\omega_1,\omega_2}$ is a $2$-orbit of $H$. Fix a permutation $g\in G_{\omega_1\rightarrow \omega_2}$ and write $r_{\omega_1,\omega_2} = g r_{\omega_1,\omega_1}$. Obviously $r_{\omega_1,\omega_1}=1_{\Omega_1}$ where $\Omega_1$ is the fixed-point set $\mathsf{Fix}(G_{\omega_1})$ of $\Omega$ under $G_{\omega_1}$. Clearly $\Omega_1$ is $H$-invariant implying that $1_{\Omega_1}$ is $H$-invariant as well.

		Observe that the action of $H$ on $1_{\Omega_1}$ is semiregular.
        Otherwise, given any $\omega\in H$, if there were to exist $h,\omega',\omega''\in H$ such that $h(\omega)=\omega$ and $h(\omega')=\omega''$, then for any element $g\in G_{\omega\to\omega'}$ we would obtain the contradiction $g(\omega)=hgh^{-1}(\omega)= h(g(\omega))=\omega''$.
		
		Finally, it is well known that $|\Omega_1|=[N_G(G_{\omega_1}):G_{\omega_1}] = |C_{\sym{\Omega}}(G)|=|H|$. Hence $H$ acts transitively in $1_{\Omega_1}$, and therefore it also acts transitively in $g1_{\Omega_1}=r_{\omega_1,\omega_2}$. We conclude that  $r_{\omega_1,\omega_2}$ is a $2$-orbit of $H$.
	\end{proof}
	
	\begin{prop}\label{WL-loop}
		Let	$(S,\diamond)$ be a loop, and let $G\le Sym(S)$ be the subgroup generated by $\{\ell_s\}_{s\in S}$. Then $C_{\sym{S}}(G)$ 
		coincides with the subgroup   $\widehat{N}_\rho:=\{r_s\}_{s\in N_\rho}$ of $\sym{S}$, where $N_\rho$ is the right nucleus of $(S,\diamond)$.
	\end{prop}
	\begin{proof}
		Recall that $N_\rho:=\{z\in S \mid \forall_{x,y\in S}\ x\diamond (y\diamond z) = (x\diamond y)\diamond z\}$. It is easy to see that $N_\rho$ contains $1_\Omega$ and is $\diamond$-closed. Together with $r_{y\diamond z} = r_y r_z$, we obtain that $\widehat{N}_\rho$ is a subgroup of $\sym{S}$. Note that $\widehat{N}_\rho$ acts semiregularly on $S$.
		
		A permutation $\phi\in \sym{S}$ centralizes $G$ if and only if it commutes with $\ell_s$ for every $s\in S$, that is, if $\phi(s\diamond u)=s\diamond \phi(u)$ for all $s,u\in S$. Substituting $u=1_\Omega$ into this equality, we obtain $\phi = r_{\phi(1_\Omega)}$, i.e.\ $\phi(x)=x\diamond \phi(1_\Omega)$. Thus the equality $\forall_{s,u\in S}\  \phi(s\diamond u)=s\diamond \phi(u)$ becomes 
		$\forall_{s,u\in S}\ (s\diamond u)\diamond \phi(1_\Omega) = 
		s\diamond (u\diamond \phi(1_\Omega))$ whence $\phi(1_\Omega)\in N_\rho$. To summarize, $\phi\in C_{\sym{S}}(G)$ implies $\phi=r_{\phi(1_\Omega)}$ and  $\phi(1_\Omega)\in N_\rho$. 

        On the other hand, if $s\in N_\rho$ then 
        \[
        \forall_{u,v\in S}\ r_s(\ell_u(v)) =(u\diamond v)\diamond s = u\diamond(v\diamond s) = \ell_u(r_s(v)),
        \]
        in other words $r_s\ell_u = \ell_u r_s$. 
	\end{proof}

	\subsection{Jordan algebraic automorphisms of coherent configurations}
	
	Let $\mathfrak{C}=(\Omega,C)$ be a coherent configuration, and let $\cA$ be its adjacency algebra which we consider as a Jordan algebra. Denote by $\Omega_1,\dots,\Omega_k$ the fibers of $\mathfrak{C}$. Then the matrices $\und{1_{\Omega_i}},i=1,\dots,k$ are the only minimal idempotents with respect to $\star$ and $\circ$. Therefore, every  algebraic automorphism $\phi$ of $\mathfrak{C}$ permutes the relations $1_{\Omega_i},i=1,\dots,k$, which we may regard as a permutation of the indices $1,\dots,k$. Following this, we abuse notation by writing
	$\phi(1_{\Omega_i}) = 1_{\Omega_{\phi(i)}}$.

    In what follows, if $r$ is a relation on  $\Omega$ and $\Omega_1\subseteq\Omega$, then we write $r\restr_{\Omega_1}:= r\cap(\Omega_1\times\Omega_1)$ to indicate the restriction of $r$ to $\Omega_1$. Moreover, if $R$ is a set of relations on $\Omega$ we write $R\restr_{\Omega_1}:=\{r\restr_{\Omega_1}\mid r\in R\}$.
	\begin{prop}\label{jaut}
		Let $\mathfrak{C}=(\Omega,C)$ be the $2$-orbit configuration of a semiregular subgroup $H\leq\sym{\Omega}$. Then 
		\begin{enumerate}[(a)]
			\item $\mathfrak{C}\cong (\Omega_1,C\restr_{\Omega_1})\otimes \mathfrak{D}_k$ where $\Omega_1$ is an $H$-orbit, $k$ is the number of $H$-orbits on $\Omega$, and $\mathfrak{D}_k$ is the discrete configuration on $[k]:=\{1,\dots,k\}$.
			\item $\jaut(\mathfrak{C})=\taut(\mathfrak{C})$.
		\end{enumerate}
	\end{prop}
	\begin{proof}
		{\sl Proof of \ref{conj_a}:}
		Let $\Omega_1,\dots,\Omega_k$ be a complete set of $H$-orbits on $\Omega$. Clearly $|\Omega_i|=|H|$ for all $i\in [k]$. Note that $\Omega_1\times [k]$ is the point set 
		of $(\Omega_1,C\restr_{\Omega_1})\otimes \mathfrak{D}_k$. Thus, 
		to build a bijection $f:\Omega_1\times [k] \rightarrow \Omega$ we choose arbitrary  $\omega_i\in \Omega_i,i\in [k]$.  Define $f(\omega,i) = h(\omega_i)$ where $h\in H$ is the unique element satisfying $\omega = h(\omega_1)$.
		
		We now check that $f$ maps basic relations of $(\Omega_1,C\restr_{\Omega_1})\otimes \mathfrak{D}_k$ onto basic relations of $\mathfrak{C}$. Since in both configurations all basic relations have the same cardinality, namely $|H|$, it is sufficient to  prove  the following holds for any quadruple of points $(\alpha,i),(\beta,j),(\alpha',i'),(\beta',j')\in \Omega_1\times [k]$:
		\[
        r((\alpha,i),(\beta,j))=r((\alpha',i'),(\beta',j'))\Rightarrow  r(f(\alpha,i),f(\beta,j))=r(f(\alpha',i'),f(\beta',j')).
		\]
		It follows from the first equality that $i'=i$ and $j'=j$. Also  $\alpha' = h(\alpha)$ and $\beta'=h(\beta)$ where $h$ is the unique element of $H$ that maps $\alpha$ to $\alpha'$. Let $a,b\in H$ be elements satisfying $\alpha=a(\omega_1),\beta=b(\omega_1)$ from which $\alpha'=ha(\omega_1)$ and $\beta'=hb(\omega_1)$ follows. Then we obtain
		$ f(\alpha,i) =a(\omega_i)$, $f(\beta,j)=b(\omega_j)$ and 
		$ f(\alpha',i) =ha(\omega_i)$, $f(\beta',j)=hb(\omega_j)$. The latter equality implies $(f(\alpha',i),f(\beta',j))=(h(f(\alpha,i)),h(f(\beta,j)))$, whence we have  $r(f(\alpha,i),f(\beta,j))=r(f(\alpha',i'),f(\beta',j'))$ proving (a).
        \smallskip
		
		{\sl Proof of \ref{conj_b}:} We first  describe the group of isomorphisms $\mathsf{Iso}(\mathfrak{C})$ (referred to as color automorphisms in~\cite{KliMuzPecWolZie07}). Since $H$ acts semiregularly on $\Omega$, $H$ is $2$-closed~\cite{Wie69}. In this case $\mathsf{Iso}(\mathfrak{C})$ coincides with the normalizer of $H$ in the symmetric group $\sym{\Omega}$.   
		
         Under the identification 
        $\mathfrak{C}$ with $(\Omega_1,C\restr_{\Omega_1})\otimes \mathfrak{D}_k$.
        a semiregular action of $H$ on $\Omega_1\times [k]$ is given by $h(\omega,i)=(h(\omega),i)$. Denoting by $F$ the set of basic relations of $(\Omega_1,C\restr_{\Omega_1})$, we can now describe the basic relations of $\mathfrak{C}$ in the  following manner: $[f]_{i,j}=\{((\alpha,i),(\beta,j)) \mid (\alpha,\beta)\in f\}$ where $f\in F$. In what follows, we set $C:=\{[f]_{ij} \mid f\in F,i,j\in [k]\}$.
        
        Since $H$ acts regularly on $\Omega_1$, its 2-orbit scheme is thin, i.e.\ each $f\in F$ is a permutation on $\Omega_1$ and $F$ is a regular permutation subgroup of $\sym{\Omega_1}$. Note that $F$ 
		is the  centralizer of $H$ in $\sym{\Omega_1}$ and $F\cong H$. 

        We next describe certain elements of the group $\aaut(\mathfrak{C})$. 
        Given $v\in F^k$, we define a permutation $\hat{v}\in\sym{C}$ by $\hat{v}: [f]_{i,j}\mapsto [v_i^{-1}fv_j]_{i,j}$. It is easy to check that $\hat{v}$ is an algebraic automorphism of $\mathfrak{C}$. Another class of algebraic automorphisms comes from the symmetric group $S_k=\sym{[k]}$, namely every permutation $\psi\in S_k$ permutes the basic relations in the following manner: $[f]_{i,j}\mapsto [f]_{\psi(i),\psi(j)}$. We denote such induced permutations by $\hat{\psi}$. 
        
        A third class of algebraic automorphisms comes from $\aut{F}$, namely 
        $\psi\in \aut{F}$  permutes the basic relations by $[f]_{i,j}\mapsto [\psi(f)]_{i,j}$.
        In keeping with the above, we denote these induced permutations by $\hat{\psi}$. One can easily check that the three types of automorphisms just described generate a group isomorphic to $F^k\rtimes (S_k\times\aut{F})$.

		Pick an arbitrary $\phi\in \jaut(\mathfrak{C})$. Then $\forall_{i\in [k]}\ \phi(I_{\Omega_i})=I_{\Omega_{{\varphi}(i)}}$ for a suitably chosen ${\varphi}\in S_k$. Since $\hat{\varphi}\in \aaut(\mathfrak{C})$, we have that $\hat{\varphi}^{-1}\phi$ is a Jordan automorphism of $\mathfrak{C}$ that fixes each fiber $I_{\Omega_i},i\in [k]$. Thus, in what follows we may assume $\forall_{i\in [k]}\ \phi(I_{\Omega_i})=I_{\Omega_i}$. 
        
        Consider now the matrices $E_{i,j}:=\frac{1}{|H|}\und{\Omega_i\times \Omega_j}$. 
		The linear span $\cA_0:=\langle E_{i,j}\rangle_{i,j\in [k]}$ is a two-sided ideal of $\cA:=\mathbb{C}[C]$ isomorphic to $M_k(\mathbb{C})$. Since $E_{i,i}$ is the unique nonzero solution to the equations $X\star I_{\Omega_i} = X$, $X\star X = X$, $X\circ X = \frac{1}{|H|} X$ and $\phi(I_{\Omega_i}) = I_{\Omega_i}$, we obtain $\phi(E_{i,i})=E_{i,i}$ for each $i \in [k]$. Given $i\neq j\in [k]$, the equations $E_{i,i}\star X =\half X$, $E_{j,j}\star X =\half X$ and $X\circ X = \frac{1}{|H|} X$ have two nonzero solutions $E_{i,j}$ and $E_{j,i}$. Hence for any pair $i\neq j$, it is either the case that $\phi(E_{i,j})=E_{j,i},\phi(E_{j,i})=E_{i,j}$ or $\phi(E_{i,j})=E_{i,j},\phi(E_{j,i})=E_{j,i}$. Let us call a pair \emph{switched} in the former case and \emph{non-switched} in the latter case.

		 Our next step is to show that if a single pair $\{i,j\}\subseteq [k]$ is switched, then all pairs in $[k]$ are switched.   
		This is trivial if $k=2$. For $k\geq 3$, it suffices to prove that given any three distinct elements $i,j,\ell\in [k]$, if $\{i,j\}$ is switched then
        $\{i,\ell\}$ and $\{j,\ell\}$ are also switched. 
        
		By way of  contradiction, suppose $\{i, j\}$ is a switched pair and at least one of $\{i,\ell\}$ and $\{j,\ell\}$ is non-switched. First suppose $\{i,\ell\}$ is non-switched. Then  $\phi(E_{i,j})=E_{j,i},\phi(E_{j,i})=E_{i,j}$ and $\phi(E_{i,\ell})=E_{i,\ell}, \phi(E_{\ell,i})=E_{\ell,i}$.  This implies 
		\[
        E_{\ell,j} = 2 E_{i,j}\star E_{\ell,i} = 2\phi(E_{j,i})\star \phi(E_{\ell,i}) = 2\phi(E_{j,i}\star E_{\ell,i}) = O,
        \]
		a contradiction. Thus we are in the case where $\{i,j\}$ and $\{i,\ell\}$ are switched and $\{j,\ell\}$ is non-switched.  But here, just as above, we reach the contradiction
        \[
        E_{i,j} = 2 E_{i,\ell}\star E_{\ell,j} = 2\phi(E_{\ell,i})\star \phi(E_{\ell,j}) = 2\phi(E_{\ell,i}\star E_{\ell,j}) = O.
        \] 
        We conclude that every pair in $[k]$ is switched or every such pair is non-switched.

        It follows from above that $\phi$ acts on $\cA_0\cong M_k(\mathbb{C})$ either trivially or via matrix transposition. Combining $\phi$ with $\tau$ in the latter case, we may assume $\phi$ acts trivially on $\cA_0$, i.e.\ $\phi(E_{i,j}) = E_{i,j}$ for all $i,j\in [k]$. This implies the set of relations $\{[f]_{i,j}\}_{f\in F}$ is $\phi$-invariant for any $i,j\in [k]$, i.e.\ there exists a permutation $\phi_{ij}\in\sym{F}$ such that $\phi([f]_{i,j}) = [\phi_{ij}(f)]_{i,j}$. 
		
		It follows from $2[f]_{ii}\star [h]_{ii} = [fh]_{ii}+[hf]_{ii}$ that $\phi_{ii}(fh)\in\{\phi_{ii}(f)\phi_{ii}(h),\phi_{ii}(h)\phi_{ii}(f)\}$. Thus $\phi_{ii}$ is a half-automorphism~\cite{Sco57} of $F$. According to~\cite{Sco57}, $\phi_{ii}$ is either an automorphism or an anti-automorphism. 
		
		For every pair $i\neq j$ and $f,h\in F$ we have $2[f]_{i,j}\star [h]_{j,i}=[fh]_{i,i}+[hf]_{j,j}$.  Applying $\phi$ to  both sides, we obtain 
		\begin{equation}\label{phi}
			\begin{array}{rcl}
				\phi_{ij}(f)\phi_{ji}(h) & = & \phi_{ii}(fh)\\
				\phi_{ji}(h)\phi_{ij}(f) & = & \phi_{jj}(hf)
			\end{array}
		\end{equation}
		for all  $f,h\in F$.
        Substituting $f=e$ into~\eqref{phi}
        now gives 
        \[ 
        \phi_{ij}(e)\phi_{ji}(h)=\phi_{ii}(h), \phi_{ji}(h)\phi_{ij}(e)=\phi_{jj}(h)\Rightarrow \phi_{ij}(e)^{-1}\phi_{ii}(h)\phi_{ij}(e) = \phi_{jj}(h).
        \]
		It follows from $f_{ii}(e)=e$ that $\phi_{ij}(e)=\phi_{ji}(e)^{-1}$.
         Substituting $h=e$ into~\eqref{phi} we obtain $\phi_{ij}(f)=\phi_{ii}(f)\phi_{ji}(e)^{-1}$. Together with $\phi_{ji}(h)=\phi_{ij}(e)^{-1}\phi_{ii}(h)$ this yields 
         \[
         \phi_{ij}(f)\phi_{ji}(h)=\phi_{ii}(f)\phi_{ji}(e)^{-1}\phi_{ij}(e)^{-1}\phi_{ii}(h) = \phi_{ii}(f)\phi_{ii}(h). 
         \]
         It now follows from~\eqref{phi} that $\phi_{ii}(f)\phi_{ii}(h) = \phi_{ii}(fh)$, i.e.\ $\phi_{ii}$ is an automorphism of $F$ for every $i\in [k]$.

        Since $\phi_{11}^{-1}\in \aut{F}$, it induces an algebraic automorphism $\widehat{\phi_{11} ^{-1}}\in\aaut(\mathfrak{C})$ which acts via $[f]_{ij}\mapsto [\phi_{11}^{-1}(f)]_{ij}$. Replacing $\phi$ by $\phi \widehat{\phi_{11}^{-1}}$,we may assume that $\phi_{11}$ is trivial. This yields $\phi_{jj}(h)=\phi_{j1}(e)^{-1}\phi_{11}(h)\phi_{1j}(e).$
    Setting $a_j:=\phi_{1j}(e)\in F$, we may write  $\phi_{jj}(h)=a_j^{-1}ha_j$ for each $j\in[k]$. Also $\phi_{1j}(f) = \phi_{11}(f) \phi_{j1}(e)^{-1} = f a_j$.

    To complete the proof, it suffices to show that $\phi = \hat{a}$ where $a:=(a_1,\dots,a_k)\in F^k$. The equality $\phi = \hat{a}$ is equivalent to 
	$\phi([f]_{ij})=\hat{a}([f]_{ij})=[a_i^{-1}fa_j]_{ij}$.
		
	It follows from~\eqref{phi} that $\phi_{ij}(f)=\phi_{ii}(f)\phi_{ji}(e)^{-1}=a_i^{-1} f a_i \phi_{ji}(e)^{-1}$. So it remains to show that $\phi_{ji}(e)=a_j^{-1}a_i$. If  $i=j$, this equality follows from $\phi_{ii}\in\aut{F}$.  Hence we assume $i\neq j$. But now   $2[e]_{ji}	= [e]_{j1}\star [e]_{1i}$ implies $\phi_{ji}(e)=\phi_{j1}(e)\phi_{1i}(e)=\phi_{1j}(e)^{-1}\phi_{1i}(e)=a_j^{-1}a_i$ as desired.
\end{proof}
\subsection{Thin regular Jordan schemes are \gen}
	
  Before proceeding to fulfill the task reflected in the subsection title, we make some small notational changes. Namely, elements of a loop $(S,\diamond)$ shall now be denoted by $\hat{s}$ rather than $\ell_s$, which was our previously chosen notation. Also, the set of all permutations $\{\hat{s}\}_{s\in S}$ will be denoted by $\hat{S}$. 
  
  We are now equipped to prove the main result of this subsection.
      \begin{theorem}
		Let $(S,\diamond)$ be a nonassociative RA-loop. Then $\mathfrak{S}:=(S,\hat{S})$ is an \gen Jordan scheme. 
	\end{theorem}
	\begin{proof}
		By way of contradiction, assume $\mathfrak{S}$ is the algebraic fusion of a coherent configuration $\mathfrak{C}=(S,C)$ with respect to a certain subgroup $\Phi\leq\jaut(\mathfrak{C})$. Then $\hat{s} =\Phi(c)$ for every $s\in S$ and $c\in C$ contained in $\hat{s}$. In particular, this implies that each $c\in C$ is thin. By \cite[Exercise 2.7.35]{ChePon24}, every such configuration is Schurian, i.e.\ there exists $H\leq\sym{S}$ such that $C$ is the set of $2$-orbits of $H$. Since every $2$-orbit of $H$ is thin, we have that $H$ is semiregular. Thus by Proposition~\ref{jaut}, $\jaut(\mathfrak{C})=\taut(\mathfrak{C})$. Hence the subgroup $\Phi_0:=\Phi\cap\aaut(\mathfrak{C})$ has index at most two in $\Phi$. Moreover, since $\Phi_0\leq \aaut(\mathfrak{C})$, the rainbow ${}^{\Phi_0}\mathfrak{C}=(S,C_0)$ where $C_0:=\{\Phi_0(c)\}_{c\in C}$ is a  coherent configuration. 
        Thus $\hat{S}={}^{\Phi}C\sqsubseteq C_0 \sqsubseteq C$. Since  $\mathfrak{S}$ is not a \CC, the inclusion $\hat{S}\sqsubseteq C_0$ must be proper, and $[\Phi:\Phi_0]=2$. 
        In particular, $\Phi_0$ is normal in $\Phi$, which in turn implies that each $\Phi$-orbit on $C$ is either a $\Phi_0$-orbit or a disjoint union of two $\Phi_0$-orbits of equal size.  
        Therefore, each $\hat{s}\in\hat{S}$ is either a basic relation of ${}^{\Phi_0}\mathfrak{C}$ or the union of two basic relations of ${}^{\Phi_0}\mathfrak{C}$ of equal size. Since ${}^{\Phi_0}\mathfrak{C}$ and ${}^{\Phi}\mathfrak{C}$ are distinct, there exists  $\hat{s}\in\hat{S}$ which splits into a disjoint union of two basic relations of ${}^{\Phi_0}\mathfrak{C}$ of equal size, say $\hat{s}=c_1\cup c_2$. It follows from $|c_1|=|c_2|=|S|/2$ that both $c_1$ and $c_2$ are nonregular basic relations of ${}^{\Phi_0}\mathfrak{C}$. Therefore ${}^{\Phi_0}\mathfrak{C}$ is a nonhomogeneous \CC, that is, each of its basic relations is nonregular. This in turn implies that all basic relations of ${}^{\Phi_0}\mathfrak{C}$ have size $|S|/2$.

		 Observe that the coherent closure $\mathfrak{T}=(S,T)$ of $\mathfrak{S}$ satisfies $\hat{S}\sqsubset T \sqsubseteq C_0$. By Proposition~\ref{WL-loop}, the basic relations of $T$ coincide with the $2$-orbits of the right nucleus $N_\rho$ of the loop $(S,\diamond)$. In particular, $|t|=|N_\rho|$ for each $t\in  T$. It follows from Proposition~\ref{loop-prop} that $N_\rho = Z(S)$, and therefore $|N_\rho|=|S|/8$. Since each $t\in T$ is the union of basic relations of $C_0$, a basic relation $c\in C_0$ has cardinality at most $|S|/8$, contrary to $|c|=|S|/2$. As this is an obvious contradiction, the proof is complete.
	\end{proof}

    \section*{Acknowledgments}
    The author MM wishes to thank Villanova University for the hospitality shown to him during the final stages of writing this paper. Additionally, the authors are grateful to R.~Lipyanski who informed them of the paper~\cite{Sco57}, and to Misha Klin for helpful discussions and suggestions. Finally, the authors acknowledge the many thoughtful comments made by the referees which greatly enhanced the quality of our paper.

\bibliographystyle{habbrv}
\bibliography{mainArXiv}

\end{document}